\documentclass[reqno]{amsart}
\usepackage{amssymb,amsthm,amsfonts,amstext}
\usepackage{amsmath}

\usepackage{graphicx}
\usepackage[latin1]{inputenc}
\usepackage{amsmath,amsthm,amsfonts,amssymb,amstext,amsfonts,amscd,latexsym,setspace}
\usepackage{mathrsfs} 
\usepackage[english]{babel}
\usepackage{graphicx}
\usepackage{color}
\usepackage{enumerate}
\DeclareFixedFont{\fiverm}{OT1}{cmr}{m}{n}{5pt}
\input{prepictex}
\input{pictex}
\input{postpictex}
\makeatletter \@addtoreset{equation}{section} \makeatother

\renewcommand\thefigure{\thesection.\@arabic\c@figure}
\renewcommand\thetable{\thesection.\@arabic\c@table}
\newtheorem{theorem}{Theorem}[section]
\newtheorem{lemma}[theorem]{Lemma}
\newtheorem{proposition}[theorem]{Proposition}
\newtheorem{corollary}[theorem]{Corollary}

\newtheorem{remark}[theorem]{Remark}

\title[E.F for a nongradient energy conserving model]{Equilibrium fluctuations for a nongradient energy conserving stochastic model} 

\begin{document}

\author{Freddy Hern\'andez}

\keywords{ Equilibrium fluctuations, Boltzmann-Gibbs principle, nongradient method, generalized 
Ornstein-Uhlenbeck process.}
\thanks{{\scriptsize CEREMADE, UMR
CNRS 7534, Universit\'e Paris-Dauphine, Place du Marechal de Lattre de Tassigny, 75775
Paris Cedex 16, FRANCE \&  IMPA, Estrada Dona Castorina 110, J. Botanico, 22460 Rio de Janeiro, Brazil.} 
}

\begin{abstract}
In this paper we study the equilibrium energy fluctuation field of a one-dimensional reversible non 
gradient model. We prove that the limit fluctuation process is governed by a generalized  Ornstein- 
Uhlenbeck process,  which covariances are given in terms of the diffusion coefficient. 

Adapting the non gradient method introduced by Varadhan,  we are able to derive the diffusion coefficient.  
The fact that the conserved quantity (energy) is not a linear functional of the coordinates of the system, 
introduces new difficulties of geometric nature when applying the nongradient method. 
\end{abstract}

\maketitle

\section{Introduction}
In recent works, a microscopic model for heat conduction in solids has been considered (\textit{c.f.} \cite{BO}, \cite{B},\cite{FNO}). In this model nearest neighbor atoms interact as coupled oscillators  forced by an additive noise which exchange kinetic energy between nearest neighbors. 

More precisely, in the case of periodic boundary conditions, atoms are labeled by $x \in \mathbb{T}_N=\{1,\cdots,N\}$. The configuration space is defined by $\Omega^N=(\mathbb R \times \mathbb R)^{\mathbb{T}_N}$, where for a typical element $(p_x,r_x)_{x \in\mathbb{T}_N}\in \Omega^N$, $r_x$ represents the distance between particles $x$ and $x+1$, and $p_x$ the velocity of the particle $x$. The formal generator of the system reads as $\mathcal{L}_N=\mathcal{A}_N+\mathcal{S}_N$, where
\begin{equation}
\mathcal{A}_N=\sum_{x\in\mathbb{T}_N}\{(p_{x+1}-p_x)\partial_{r_x}+(r_x-r_{x-1})\partial_{p_x}\}
   \;,
\end{equation}
and
\begin{equation}
\label{gradient}
\mathcal{S}_N=\frac{1}{2}\sum_{x\in\mathbb{T}_N}X_{x,x+1}[X_{x,x+1}]
   \;,
\end{equation}
with $X_{x,x+1}=p_{x+1}\partial_{p_x}-p_{x}\partial_{p_{x+1}}$. Here $\mathcal{A}_N$ is the Liouville operator of a chain of interacting harmonic oscillators and $\mathcal{S}_N$ is the noise operator.

In this work we focus on the noise operator $\mathcal S_N$, which acts only on velocities. Therefore, we restrict the configuration space to $\mathbb R^{\mathbb{T}_N}$. The total energy of the configuration $(p_x)_{x \in \mathbb T_N}$ is defined by
\begin{equation}
\label{totalenergy2}
\mathcal E =\frac{1}{2}\sum_{x\in\mathbb{T}_N}p_x^2\;.
\end{equation}
It is easy to check that $\mathcal{S}_N(\mathcal E)=0$, \textit{i.e} total energy is constant in time. 

The generator $\mathcal{S}_N$ defines a diffusion process with invariant measures given by  $\nu^N_y(dp)= \otimes_{x \in \mathbb{T}_N} \frac{1}{\sqrt {2\pi}y}e^{-p^2_x/2y^2}dp_x$ for all $y>0$. This process is not ergodic with respect to these measures, in fact,  for all $\beta>0$ the hyperspheres $p_1^2+\cdots+p_N^2=N\beta$ of average kinetic energy $\beta$ are invariant sets. Nevertheless, the restriction of the diffusion to each of these hyperspheres is nondegenerate and ergodic.

In analogy to \cite{V} (\textit{see also} \cite{Q}) , where Varadhan introduced the nongradient method, we introduce inhomogeneities into the diffusion generated by (\ref{gradient}) through a differentiable function $a(r,s)$ satisfying $0 < c\leq a(r,s)\leq C<\infty$ and having bounded continuous first derivatives (\textit{see} (\ref{infinitgenerator})). As a result, the introduction of the function $a(r,s)$ breaks the gradient structure this diffusion. 

The main result of this work is convergence of the energy fluctuation field defined in (\ref{eqflucfield}) to a generalized Ornstein-Uhlenbeck process,  when the process is at equilibrium. The covariances characterizing this generalized process are given in terms of the diffusion coefficient  $\hat{a}(y)$ (\textit{see} (\ref{Orns-Uhlen})). This diffusion coefficient is given in terms of a variational formula which is equivalent to the Green-Kubo formula (\textit{c.f.} \cite{Sp} p.180).  The main task of this work is to establish rigorously this variational formula.

In order to study the equilibrium fluctuations of interacting particle systems, Brox and Rost \cite{BR} introduced the Boltzmann-Gibbs principle and proved its validity for attractive zero range processes. Chang and Yau \cite{CY} proposed an alternative method to prove the Boltzmann-Gibbs principle for \textit{gradient systems}. This approach was extended to \textit{nongradient systems} by Lu \cite{Lu} and Sellami \cite{Se}.  

In what follows we describe the main features of the model we consider.

\textbf{The model is non gradient.} This difficulty has already appeared in the work of Bernardin \cite{B}, where there are two conserved quantities (total deformation and total energy). The energy current is not the gradient of a local function. To overcome this problem, an exact fluctuation-dissipation relation is obtained; that is, the current is written as a gradient plus a fluctuation term. On the other hand, in \cite{FNO}  Fritz \textit{et al} studied the equilibrium fluctuations for the model given in \cite{B}. The exact fluctuation-dissipation relation mentioned above plays a central role in the proofs of the hydrodynamic limit and the equilibrium fluctuations.
 
Systems for which there exists an exact fluctuation-dissipation relation are called \textit{almost gradient systems}. For this kind of systems, the minimizer in the variational formula of the diffusion coefficient can be found explicitly. In our setting we do not have such an exact relation, so we use the nongradient  Varadhan's method.

\textbf{The only conserved quantity (total energy) is not a linear function of the coordinates of the system.} In other words, the invariant surfaces are not hyperplanes. Specifically, in our case invariant surfaces are hyperspheres. 

In the non gradient  Varadhan's method, it is central to have a characterization of the space over which the infimum in the variational problem defining the diffusion coefficient is taken. In order to obtain such characterization, some results related to differential forms on spheres and  integration over spheres are needed. 

\textbf{We do not have good control when dealing with large velocities.} 
This lack of control makes  the estimation of  exponential moments difficult. In \cite{B} the author manages to overcome this difficulty by adopting a microcanonical approach.  Estimation of  exponential moments arises in our case when trying to adapt the usual proof of tightness. Using the microcanonical approach mentioned before, lead us to an identity we are unable to prove. This identity is in fact equivalent to the one conjectured by Bernardin  (\cite{B}, lemma 6.3). To avoid the exponential estimate, we exploits the fact that Boltzmann-Gibbs principle can be interpreted as an asymptotic gradient condition (as pointed out in \cite{CLO}). 

Let us finish by explaining how this paper is organized. By adapting the method introduced in \cite{V} we identify the diffusion term (Section \ref{CLTVandDC}), which allows us to derive the Boltzmann-Gibbs principle (Section \ref{Boltzmann-Gibbs}). 
This is the key point to show that the energy fluctuation field  converges in the sense of finite dimensional distributions to a generalized Ornstein-Uhlenbeck process (Section \ref{finite-dimensional distributions}). Moreover, using again the  Boltzmann-Gibbs principle we also prove tightness for the energy fluctuation field in a specified Sobolev space (Section \ref{Tightness}), which together with the finite dimensional convergence implies the convergence in distribution to the generalized Ornstein-Uhlenbeck process mentioned above. In Section \ref{{H}_y} a characterization of the space involved in the variational problem defining the diffusion coefficient is given. This characterization relies on a sharp spectral gap estimate (Appendix \ref{spectralgap}) and some integrability conditions for Poisson systems  studied in Appendix \ref{geomconsider}. For the sake of completeness we state in Appendix \ref{equiensem} an equivalence of ensembles result.  

\section{Notations and Results}
\label{NotandRes}
We will now give a precise description of the model. We consider a system of $N$ particles in one dimension evolving under an interacting random mechanism. It is assumed that the spatial distribution of particles is uniform, so that the state of the system is given by specifying the $N$ velocities.

Let $\mathbb T=(0,1]$ be the 1-dimensional torus, and for a positive integer $N$ denote by $\mathbb{T}_N$ the lattice torus of length N : $\mathbb{T}_N=\{1,\cdots,N\}$. The configuration space is denoted by $\Omega^N=\mathbb{R}^{\mathbb{T}_N}$ and a typical configuration is denoted by $p=(p_x)_{x \in \mathbb{T}_N}$,  where $p_x$ represents the velocity of the particle in $x$. The velocity configuration $p$ changes with time and, as a function of time undergoes a diffusion in $\mathbb R^N$.

The diffusion mentioned above have as infinitesimal generator the following operator
\begin{equation}
\label{infinitgenerator}
\mathcal{L}_N=\frac{1}{2}\sum_{x\in\mathbb{T}_N}X_{x,x+1}[a(p_x,p_{x+1})X_{x,x+1}],
\end{equation}
where $X_{x,z}=p_{z}\partial_{p_x}-p_{x}\partial_{p_{z}}$, $a : \mathbb R^2 \to \mathbb R$ is a differentiable function satisfying $0 < c\leq a(x,y)\leq C<\infty$ with bounded continuous first derivatives. Of course, all the sums are taken \textit{modulo} $N$. Observe that the total energy defined as $\mathcal{E}^N=\frac{1}{2}\sum_{x\in\mathbb{T}_N}p_x^2$  \ satisfies $\mathcal{L}_N(\mathcal{E}^N)=0$, \textit{i.e}  total energy is a conserved quantity.

Let us consider for every $y >0 $ the Gaussian product measure $\nu_{y}^N$ on $\Omega^N$ with density relative to the Lebesgue measure given by
\begin{equation}
\label{invmeas}
\nu_{y}^N(dp)=\prod_{x\in\mathbb{T}_N}\frac{e^{-\frac{p^2_x}{2y^2}}}{\sqrt{2\pi}y} dp_x,
\end{equation}
where $p=(p_{1},p_{2},\hdots,p_{N})$. 

Denote by ${L}^2(\nu_{y}^N)$ the Hilbert space of functions $f$ on $\Omega^N$ such that $\nu_{y}^N(f^2)< \infty$. $\mathcal{L}_N$ is formally symmetric on ${L}^2(\nu_{y}^N)$. In fact, is easy to see that for smooth functions $f$ and $g$ in a core of the operator $\mathcal{L}_N$, we have for all $y>0$
\begin{align*}
\int_{\mathbb{R}^{N}} \mathcal{L}_{N}(f)g\nu_{y}^N(dp)
                            =\int_{\mathbb{R}^{N}} f\mathcal{L}_{N}(g)\nu_{y}^N(dp).
\end{align*}
In particular, the diffusion is reversible with respect to all the invariant measures $\nu^N_y$. 

On the other hand, for every $y>0$ the Dirichlet form of the diffusion with respect to $\nu_y^N$ is given by
\begin{align}
\label{Dirichletform}
\nonumber \mathcal{D}_{N,y}(f)
             =\frac{1}{2}\sum_{x\in\mathbb{T}_N}\int_{\mathbb{R}^{N}} a(p_x,p_{x+1})[X_{x,x+1}(f)]^2\nu_{y}^N(dp)\;.
\end{align}

Denote by $\{p(t),t\geq 0\}$ the Markov process generated by $N^2\mathcal{L}_N$ (the factor $N^2$ correspond to an acceleration of time). Let $C(\mathbb{R}_{+},\Omega^N)$ be the space of continuous trajectories on the configuration space. Fixed a time $T>0$ and for a given measure $\mu^N$ on $\Omega^N$, the probability measure on $C([0,T],\Omega^N)$ induced by this Markov process starting in $\mu^N$  will be denoted by $\mathbb{P}_{\mu^N}$. As usual, expectation with respect to $\mathbb{P}_{\mu^N}$ will be denoted by $\mathbb{E}_{\mu^N}$.

The diffusion generated by $N^2\mathcal{L}_N$ can also be described by the
following system of stochastic differential equations
{\small
\begin{align*}
dp_x(t)&=\frac{N^2}{2}\{X_{x,x+1}[a(p_x,p_{x+1})]p_{x+1}-X_{x-1,x}[a(p_{x-1},p_{x})]p_{x-1}-p_{x}[a(p_x,p_{x+1})\\
       &+a(p_{x-1},p_x)]\}dt+N[p_{x-1}\sqrt{a(p_{x-1},p_{x})}dB_{x-1,x}-p_{x+1}\sqrt{a(p_x,p_{x+1})}dB_{x,x+1}],
\end{align*}
}
where $\{B_{x,x+1}\}_{x\in\mathbb{T}_N}$ are independent standard
Brownian motion. 

Then, by It\^o's formula we have that
{\small
\begin{equation}
\label{itoform}
dp^2_x(t)=N^2[W_{x-1,x}-W_{x,x+1}]dt+N[\sigma(p_{x-1},p_{x})dB_{x-1,x}(s)-\sigma(p_x,p_{x+1})dB_{x,x+1}(s)],
\end{equation}
}
where,
\begin{equation}
\label{current}
W_{x,x+1}=a(p_x,p_{x+1})(p^2_x-p^2_{x+1})-X_{x,x+1}[a(p_x,p_{x+1})]p_{x}p_{x+1}\;,
\end{equation}
and,
\begin{equation}
\sigma(p_x,p_{x+1})=2p_xp_{x+1}\sqrt{a(p_x,p_{x+1})}\;.
\end{equation}
We can think of $W_{x,x+1}$ as being the instantaneous microscopic current of energy between $x$ and $x+1$. Observe that the current $W_{x,x+1}$ cannot be written as the gradient of a local function, neither by an exact fluctuation-dissipation equation, \textit{i.e}  as the sum of a gradient and a dissipative term of the form $\mathcal{L}_N(\tau_x h)$. That is, we are in the \textit{nongradient}  case. 

The collective behavior of the system is described thanks to empirical measures. With this purpose let us introduce the energy empirical measure associated to the process defined by 
\[
\pi^N_t(\omega,du)=\frac{1}{N}\sum_{x\in\mathbb{T}_N}p^2_x(t)\delta_{\frac{x}{N}}(du)\;,
\]
where $\delta_u$ represents the Dirac measure concentrated on $u$.
 
To investigate equilibrium fluctuations of the empirical measure $\pi^N$ we fix once and for all $y>0$ and consider the system in the equilibrium $\nu_y^N$. Denote by  $Y_t^N$  the empirical energy fluctuation field acting on smooth functions $H:\mathbb T \to \mathbb R$ as
\begin{equation}
\label{eqflucfield}
Y_t^N(H) = \frac{1}{\sqrt{N}} \sum_{x \in \mathbb T_N}H(x/N)\{p_x^2(t)-y^2\}\;.
\end{equation}
On the other hand, let $\{Y_t\}_{t \geq 0}$ be the stationary generalized Ornstein-Uhlenbeck process with zero mean and covariances given by
\begin{equation}
\label{Orns-Uhlen}
\mathbb E[Y_t(H_1)Y_s(H_2)]=\frac{4y^4}{\sqrt{4\pi(t-s)\hat{a}(y)}}
\int_{\mathbb T}du\int_{\mathbb R}dv \bar{H_1}(u)\exp\left\{-\frac{(u-v)^2}{4(t-s)\hat{a}(y)}\right\} \bar{H_2}(v)\;,
\end{equation}
for every $0 \leq s \leq t$. Here $\bar{H_1}(u)$ ( resp $\bar{H_2}(u)$) is the periodic extension to the real line of the smooth function $H_1$ (resp $H_2$), and $\hat{a}(y)$ is the diffusion coefficient determined later in Section \ref{CLTVandDC}.

Consider for $k>\frac{3}{2}$ the Sobolev space  $\mathcal H_{-k}$, whose definition will be given at the beginning of Section \ref{Tightness}. Denote by $\mathbb Q_N$ the probability measure on $C([0,T],\mathcal H_{-k})$ induced by the energy fluctuation field $Y_t^N$ and the Markov process $\{p^N(t),t \geq 0\}$ defined at the beginning of this section, starting from the equilibrium probability measure $\nu^N_{y}$. Let $\mathbb Q$ be the probability measure on the space $C([0,T],\mathcal H_{-k})$ corresponding to the generalized Ornstein-Uhlenbeck process $Y_t$ defined above.

We are now ready to state the main result of this work.
\begin{theorem}
\label{equifluct}
The sequence of probability measures $\{\mathbb Q_N\}_{N \geq 1}$ converges weakly to the probability measure $\mathbb Q\;$.
\end{theorem}
The proof of Theorem \ref{equifluct} will be divided into two parts. On the one hand, in Section \ref{Tightness} we prove tightness of $\{\mathbb Q_N\}_{N \geq 1}$ , where also a complete description of the space $\mathcal H_{-k}$ is given. On the other hand, in Section \ref{finite-dimensional distributions} we prove the finite-dimensional distribution convergence. These two results together imply the desired result. 
Let us conclude this section with a brief description of the approach we follow.

Given a smooth function $H:\mathbb{T} \times [0,T] \to \mathbb{R}$, we have after (\ref{itoform}) that
\begin{equation}
\label{martingale}
\begin{split}
M_N^H(t)=Y_t^N(H_t) -Y_t^N(H_0) -\int^{t}_{0}Y_t^N(\partial_s H_s)  ds \\
      +\int^{t}_{0}\sqrt N\sum_{x\in\mathbb{T}_N}\nabla_N H(\frac{x}{N},s)W_{x,x+1}(s)ds\;, 
\end{split}
\end{equation}
where $H_t(\cdot) = H(\cdot,t)$ and the left hand side is the martingale
\begin{equation*}
M_N^H(t)=\frac{1}{\sqrt N}\sum_{x\in\mathbb{T}_N}\int^{t}_{0}\nabla_N H(\frac{x}{N},s)\sigma(p_x,p_{x+1})dB_{x,x+1}(s)\;,
\end{equation*}
whose quadratic variation is given by
\begin{equation*}
\langle M_N^H\rangle(t)=\frac{1}{N}\sum_{x\in\mathbb{T}_N}\int^{t}_{0}|\nabla_N H(\frac{x}{N},s)|^2a(p_x,p_{x+1})p^2_xp^2_{x+1}ds\;.
\end{equation*}
Here $\nabla_N$ denotes the
discrete gradient. Recall that if $H$
is a smooth function defined on $\mathbb{T}$ and $\nabla$ is the
continuous gradient, then
\[
\nabla_N H(\frac{x}{N})=N[H(\frac{x+1}{N})-H(\frac{x}{N})]=(\nabla
H)(\frac{x}{N})+o(N^{-1}).
\]
In analogy, $\Delta_N$ denotes the discrete Laplacian, which satisfies
\[
\Delta_N H(\frac{x}{N})=N^2[H(\frac{x+1}{N})-2H(\frac{x}{N})+H(\frac{x-1}{N})]=(\Delta
H)(\frac{x}{N})+o(N^{-1}),
\]
with $\Delta$ being the continuous Laplacian.

To close the equation for the martingale $M_N^H(t)$ we have to replace the term involving the microscopic currents in (\ref{martingale}) with a term involving $Y_t^N$. Roughly speaking, what makes possible this replacement is the fact that non-conserved quantities fluctuates faster than conserved ones. Since the total energy is the unique conserved quantity of the system, it is reasonable that the only surviving part of the fluctuation field represented by the last term in (\ref{martingale}) is its projection over the conservative field $Y_t^N$. This is the content of the Boltzmann-Gibbs Principle (see \cite{BR}).

Recall  that in fact we are in a \textit{nongradient} case. Therefore, in order to perform the replacement mentioned in the previous paragraph, we follow the approach proposed by Varadhan in \cite{V}. Roughly speaking, the idea is to decompose the current as a sum of a gradient term plus a fluctuation term. The key point is that such a decomposition allows to study separately the diffusive part of the current and the part coming from a fluctuation term.

\section{Convergence of the finite-dimensional distributions}
\label{finite-dimensional distributions}
We state the main result of the section.
\begin{theorem}
\label{equifluct1}
The finite dimensional distributions of the fluctuation field $Y_t^N$ defined in (\ref{eqflucfield}) converges, as $N$ goes to infinity, to the finite dimensional distributions of the generalized Ornstein-Uhlenbeck process $Y$ defined in \text{(\ref{Orns-Uhlen})}.
\end{theorem}
In this setting, convergence of finite dimensional distributions means that given a positive integer $k$, for every $\{t_1,\cdots,t_k\} \subset [0,T]$ and every collection of smooth functions $\{H_1,\cdots,H_k\}$, the vector $\left(Y_{t_1}^N(H_1),\cdots,Y_{t_k}^N(H_k)\right)$ converges in distribution to the vector $\left(Y_{t_1}(H_1),\cdots,Y_{t_k}(H_k)\right)$.

From (\ref{martingale}) we have
\begin{align*}
Y_t^N(H) =\; &Y_0^N(H) - \int^t_0 \sqrt{N} \sum_{x \in \mathbb T_N}\nabla_N H(x/N) W_{x,x+1}(s)ds - 
\end{align*} 
The idea is to use the observations made at the end of Section \ref{NotandRes}, together with $\sum_{x \in \mathbb T_N}\Delta_NH(x/N)=0$, in order to replace the integral term corresponding to the current $W_{x,x+1}$ by an expression involving the empirical energy fluctuation field, namely $\int_0^t Y_s^N(\Delta_NH)ds$.

We begin by rewriting expression (\ref{martingale}) as
\begin{align}
\label{redependingtimemart}
Y_t^N(H_t) =\; &Y_0^N(H_0) +\int^t_0 Y_s^N( \partial_sH_s +\hat{a}(y)\Delta_NH_s )ds \; - I^{1}_{N,F}(H^{\cdot t}) -I^{2}_{N,F}(H^{\cdot t}) \nonumber \\
        &-M^{1}_{N,F}(H^{\cdot t})-M^{2}_{N,F}(H^{\cdot t})\;,
\end{align} 
where $F$ is a fixed smooth local function and
{\small
\begin{align*}
I^{1}_{N,F}(H^{\cdot t})&= \int^t_0 \sqrt{N} \sum_{x \in \mathbb T_N}\nabla_N H_s(x/N) \left[W_{x,x+1}-\hat{a}(y)[p^2_{x+1}-p^2_x]-\mathcal{L}_N\tau^x F\right]ds\;,\\
I^{2}_{N,F}(H^{\cdot t})&= \int^t_0 \sqrt{N} \sum_{x \in \mathbb T_N}\nabla_N H_s(x/N) \mathcal{L}_N\tau^x Fds\;,\\
M^{1}_{N,F}(H^{\cdot t})&= \int^t_0 \frac{2}{\sqrt{N}} \sum_{x \in \mathbb T_N}\nabla_N H_s(x/N)\tau^x\sqrt{a(p_0,p_1)}\left[p_0p_1+X_{0,1}(\sum_{i \in \mathbb T_N}\tau^i F)\right]dB_{x,x+1}\;,\\
M^{2}_{N,F}(H^{\cdot t})&= \int^t_0 \frac{2}{\sqrt{N}} \sum_{x \in \mathbb T_N}\nabla_N H_s(x/N)\sqrt{a(p_x,p_{x+1})}X_{x,x+1}(\sum_{i \in \mathbb T_N}\tau^i F)\;dB_{x,x+1}(s)\;.
\end{align*}
}
Here $\tau^x$ represents translation by $x$, and the notation $H^{\cdot t}$ stressed the fact that functionals depend on the function $H$ through times in the interval $[0,t]$. Let us now explain the reason to rewrite expression (\ref{martingale}) in this way. 

In Section \ref{CLTVandDC} (see (\ref{^a1}) and (\ref{^a2})) the following variational formula for the diffusion coefficient $\hat{a}(y)$  is obtained.\begin{equation*}
\hat{a}(y)=y^{-4}\inf_{F} a(y,F)\;,
\end{equation*}
where the infimum is taken over all local smooth functions belonging to the Schwartz space, and
\begin{equation*}
a(y,F)=\textbf{E}_{\nu_y}[a(p_0,p_1)\{ p_0p_1+X_{0,1}(\sum_{x \in \mathbb Z}\tau^x F )\}^2]\;.
\end{equation*}
Let $\{\frac{1}{2}F_k\}_{k \geq  1}$ be a minimizing sequence of local functions, that is
\begin{equation}
\label{sequencetohat}
\lim_{k \to \infty} a(y,\frac{1}{2}F_k)=y^4\hat{a}(y) \;.
\end{equation}
With this notation we have the following result.
\begin{theorem}[\bf{Boltzmann-Gibbs Principle}]
\label{BoltzGibbs}
For the sequence $\{F_k\}_{k \geq  1}$ given above and every smooth function $H:\mathbb T\times [0,T]  \to \mathbb R$, we have
\begin{equation}
\lim_{k \to \infty} \lim_{N \to \infty}\mathbb E_{\nu^N_y}\left[(I^{1}_{N,F_k}(H^{\cdot t})^2\right]\;=\;0\;.
\end{equation}
\end{theorem}

On the other hand, a judicious choice of the function $H$ will cancel the second term in the right hand side of (\ref{redependingtimemart}). Let us firstly note that we can replace $\Delta_N H_s$ by $\Delta H_s$.
In fact, the smoothness of $H$ implies the existence of a constant $C >0$ such that
\[ 
|Y_s^N(\Delta H_s )- Y_s^N(\Delta_NH_s )| \leq \frac{C}{\sqrt N}\left( \frac{1}{N}\sum_{x \in \mathbb T_N} p_x^2(s)\right)\;,
\]
uniformly in $s$.  

Denote by $\{S_t\}_{t \geq 0}$ the semigroup generated by the Laplacian operator $\hat{a}(y)\Delta$. Given $t > 0$ and a smooth function $H:\mathbb T \to \mathbb R$, define $H_s=S_{t-s}H$ for $0 \leq s \leq t$. As is well known, the following properties are satisfied :
\begin{align}
\label{heat1} \partial_sH_s\;+\;\hat{a}(y)\Delta H_s\;&=\;0\;,\\
\label{heat2} \langle \;H_s \;,\;\hat{a}(y)\Delta H_s \;\rangle \;&=\; -\frac{1}{2}\frac{d}{ds}\langle\; H_s \;,\; H_s\; \rangle \;,
\end{align}
where $\left\langle \cdot,\cdot \right\rangle$ stands for the usual inner product in $L^2(\mathbb T)$.

In this way we obtain for all smooth functions $H:\mathbb T \to \mathbb R$ 
\begin{equation}
\label{itorelation}
Y_t^N(H) =\;Y_0^N(S_t H)  \;+\; O(\frac 1N) - I^{1}_{N,F}(H^{\cdot t}) -I^{2}_{N,F}(H^{\cdot t}) - M^{1}_{N,F}(H^{\cdot t})-M^{2}_{N,F}(H^{\cdot t})\;,
\end{equation} 
where $O(\frac 1N)$ denotes a function whose $L^2$ norm is bounded by $C/N$ for a constant $C$ depending just on $H$.

The following two lemmas concern the remaining terms.
\begin{lemma}
\label{twoterms1}
For every smooth function $H:\mathbb T \times [0,T] \to \mathbb R$ and local function $F$ in the Schwartz space, 
\begin{equation}
\lim_{N \to \infty} \mathbb E_{\nu^N_y}\left[\sup_{0\leq t \leq T}\left(I^{2}_{N,F}(H^{\cdot t}) + M^{2}_{N,F}(H^{\cdot t})\right)^2\right]\;=\;0\;.
\end{equation}
\end{lemma}

\begin{lemma}
\label{convermart}
The process $M^{1}_{N,F_k}$ converges in distribution as $k$ increases to infinity after $N$, to a generalized Gaussian process characterized by
\begin{equation}
\label{limitcovariace}
\lim_{k \to \infty}\lim_{N \to \infty} \mathbb E_{\nu^N_y}\left[ M^{1}_{N,F_k}(H_1^{\cdot t})M^{1}_{N,F_k}(H_2^{\cdot t})
\right]\;=\;4y^4\int_0^t\int_{\mathbb T}\hat{a}(y)H_1^{\prime}(x,s)H_2^{\prime}(x,s)\;dxds\;,
\end{equation}
for every smooth function $H_i:\mathbb T \times [0,T] \to \mathbb R$ for $i=1,2$.
\end{lemma}

The proofs of Lemma \ref{twoterms1} and Lemma \ref{convermart} are postponed to the end of this section. The proof of Theorem \ref{BoltzGibbs} is considerably more difficult, and Section \ref{Boltzmann-Gibbs} is devoted to it.

Before entering in the proof of Theorem \ref{equifluct1} we state some remarks. Firstly, the convergence in distribution of $Y_0^N(H)$ to a Gaussian random variable with mean zero and variance $2y^4\langle H , H \rangle$ as $N$ tends to infinity, follows directly from the Lindeberg-Feller theorem.  

Property (\ref{heat2}) together with Lemma \ref{convermart} imply the convergence in distribution as $k$ increases to infinity after $N$ of the martingale $ M^{1}_{N,F_k}(H^{\cdot t})$ to a Gaussian random variable with mean zero and variance  $2y^4\langle H , H \rangle - 2y^4\langle S_tH , S_tH \rangle$.

Finally, observe that the martingale $ M^{1}_{N,F_k}(H^{\cdot t})$ is independent of the initial filtration $\mathcal F_0$.

\begin{proof}
[Proof of Theorem \ref{equifluct1}] 
For simplicity and concreteness in the exposition we will restrict ourselves to the two-dimensional case $(Y_t^N(H_1),Y_0^N(H_2))$ . Similar arguments may be given to show the general case.

In order to characterize the limit distribution of $(Y_t^N(H_1),Y_0^N(H_2))$ it is enough to characterize the limit distribution of all the linear combinations $\theta_1Y_t^N(H_1)+\theta_2Y_0^N(H_2)$.

From Lemma \ref{twoterms1}, Theorem \ref{BoltzGibbs} and expression \ref{itorelation} it follows
\begin{align*}
\theta_1Y_t^N(H_1)+\theta_2Y_0^N(H_2)\;
           &=\;Y_0^N(\theta_1S_t H_1+\theta_2H_2)  + \theta_1 I^N(H_1,F_k) -\theta_1 M^{1}_{N,F_k}(H_1^{\cdot t}),
\end{align*}
where $I^N(H,F_k)$ denotes a function whose $L^2$ norm tends to zero as $k$ increases to infinity after $N$.

Thus the random variable $\theta_1Y_t^N(H_1)+\theta_2Y_0^N(H_2)$ tends, as $N$ goes to infinity, to a Gaussian random variable with mean zero and variance 
 $2y^4\{\theta_1^2\langle H_1 , H_1 \rangle +2\theta_1\theta_2\langle S_tH_1 ,H_2 \rangle+\theta_2^2\langle H_2 , H_2 \rangle\} $.
This in turn implies
\[
\mathbb E_{\nu^N_y}[Y_t(H_1)Y_0(H_2)] \;=\; 2y^4\langle S_tH_1 ,H_2 \rangle \;,
\]
which coincide with (\ref{Orns-Uhlen}) as can be easily verified by using the explicit form of $S_tH$ in terms of the heat kernel.
\end{proof}
Now we proceed to give the proofs of Lemma \ref{twoterms1} and Lemma \ref{convermart}.
\begin{proof}
[Proof of Lemma \ref{twoterms1}]
Let us define
\begin{equation*}
\zeta_{N,F}(t) = \frac{1}{{N}^ {3/2}} \sum_{x \in \mathbb T_N}\nabla_{N} H_t(x/N)\tau^ x F(p(t))\;.
\end{equation*} 
From the It\^o's formula we obtain
{\small
\begin{align*}
\zeta_{N,F}(t) - &\zeta_{N,F}(0) = \frac{1}{2}I^{2}_{N,F}(H^{\cdot t})  + \int^t_0 \frac{1}{{N}^ {3/2}} \sum_{x \in \mathbb T_N} \partial_ t \nabla_N H_s(x/N) \tau^x F(p)ds \\
   &+ \int^t_0 \frac{1}{\sqrt{N}} \sum_{x \in \mathbb T_N}\nabla_N H_s(x/N) \sum_{z \in \mathbb T_N}\sqrt{a(p_z,p_{z+1})}X_{z,z+1}(\tau^x F)\;dB_{z,z+1}(s).
\end{align*}
}
Then $(I^{2}_{N,F}(H^{\cdot t})+M^{2}_{N,F}(H^{\cdot t}))^2$ is bounded above by 6 times the following sum
$$
(\zeta_{N,F}(t) - \zeta_{N,F}(0))^2
+
\left( \int^t_0 \frac{1}{{N}^ {3/2}} \sum_{x \in \mathbb T_N} \partial_ t \nabla_N H_s(x/N) \tau^x F(p)ds \right)^2+
$$
{\small
$$
\left(\frac{1}{2}M^{2}_{N,F}(H^{\cdot t}) - \int^t_0 \frac{1}{\sqrt{N}} \sum_{x \in \mathbb T_N}\nabla_N H_s(x/N) \sum_{z \in \mathbb T_N}\sqrt{a(p_z,p_{z+1})}X_{z,z+1}(\tau^x F)\;dB_{z,z+1}(s)\right)^2\!\!\!\!.
$$
}
Since $F$ is bounded and $H$ is smooth, is easy to see that the first two terms are of order $\frac {1}{N}$. Using additionally  the fact that $F$ is local, we can prove that the expectation of the $\sup_{0\leq t \leq T}$ of the third term is also of order $\frac {1}{N}$. In fact,  if $M \in \mathbb N$ is such that $F(p) = F(p_{0}, \cdots, p_{M})$, we are considering the difference between
$$
\int^t_0 \frac{1}{\sqrt{N}}  \sum_{x \in \mathbb T_N}\nabla_N H_s(x/N)\sqrt{a(p_x,p_{x+1})} \sum_{j = 1}^M X_{x,x+1}(\tau^{x-j} F)\;dB_{x,x+1}(s)\;,
$$
and
$$
 \int^t_0 \frac{1}{\sqrt{N}} \sum_{x \in \mathbb T_N}\nabla_N H_s(x/N)\sum_{j = 0}^{M+1} \sqrt{a(p_{x-1+j},p_{x+j})}X_{x-1+j,x+j}(\tau^x F)\;dB_{x-1+j,x+j}(s).
$$
After rearrangement of the sum, last line can be written as
$$
 \int^t_0 \frac{1}{\sqrt{N}} \sum_{x \in \mathbb T_N} \sqrt{a(p_{x},p_{x+1})} \sum_{j = 0}^{M+1}\nabla_N H_s(x-j+1/N)X_{x,x+1}(\tau^{x-j+1} F)\;dB_{x,x+1}(s) \;.
$$
The proof is then concluded by using Doob's inequality.
\end{proof}
\begin{proof}
[Proof of Lemma \ref{convermart}]
Using basic properties of the stochastic integral and the stationarity of the process, we can see that the expectation appearing in the left side of (\ref{limitcovariace}) is equal to
{\small
\begin{equation*}
\int^t_0 \frac{4}{N} \sum_{x \in \mathbb T_N}
\nabla_N H_{1,s}(x/N)\nabla_N H_{2,s}(x/N)\mathbb E_{\nu^N_y}\!\!\left[ \tau^x a(p_0,p_1)\left(p_0p_1\!\!+\!\!X_{0,1}(\sum_{i \in \mathbb T_N}\tau^i F_k)\right)^2\right]ds.
\end{equation*}
}
Translation invariance of the measure $\nu^N_y$ lead us to 
{\small
\begin{equation*}  
4\;\mathbb E_{\nu^{N}_y}\left[ a(p_0,p_1)\left(p_0p_1+X_{0,1}(\sum_{i \in \mathbb T_N}\tau^i F_k)\right)^2\right]
\int^t_0
\frac{1}{N}\sum_{x \in \mathbb T_N} \nabla_N H_{1,s}(x/N)\nabla_N H_{2,s}(x/N)ds,
\end{equation*}
}
and as $N$ goes to infinity we obtain
\begin{equation*}  
4\;\mathbb E_{\nu_y}\left[ a(p_0,p_1)\left(p_0p_1+X_{0,1}(\sum_{i \in \mathbb Z}\tau^i F_k)\right)^2\right]
\int^t_0
\int_{\mathbb T} H^{\prime}_{1}(u,s)H^{\prime}_{2}(u,s)\;du\;ds\;.
\end{equation*}
Finally from (\ref{sequencetohat}), taking the limit as $k$ tends to infinity we obtain the desired result.
\end{proof}

\section{Central Limit Theorem Variances and Diffusion Coefficient}
\label{CLTVandDC}
The aim of this section is to identify the diffusion coefficient $\hat{a}(y)$, which is the asymptotic component of the current $W_{x,x+1}$ in the direction of the gradient. More precisely, $\hat{a}(y)$ will be the constant for which the infimum over all smooth local functions $F$ of the expression below vanish.  

{\small
\begin{equation}
\label{localcurrentdensity}
\limsup_{N \to \infty}\frac{1}{2N}\limsup_{t \to \infty}\frac{1}{t}  \textbf E_{\nu_y}\left[ \left(\int_0^t 
\sum_{-N\leq x\leq x+1\leq N}\! \! \! \! \! \! W_{x,x+1}-  \hat{a}(y)(p^2_{x+1}-p^2_{x}) - \mathcal L(\tau^x F)\ ds\right)^2\right].
\end{equation}
}
Here we are considering the process generated by  $\mathcal L$ and $\nu_y$, the natural extension of $\mathcal L_N$ to the infinite product space $\Omega=\mathbb{R}^\mathbb{Z}$ and the infinite product measure (\ref{invmeas}), respectively. 

The form of the limit with respect to $t$ appearing in  (\ref{localcurrentdensity}) leads us to think in the central limit theorem  for additive functionals of Markov processes. Let us begin by introducing some notations and stating some general results for continuous time Markov processes.

Consider a continuous time Markov process $\{Y_s\}_{s \geq 0}$, reversible and ergodic with respect to invariant measure $\pi$. Denote by $\langle\;,\;\rangle_{\pi}$ the inner product in $L^2(\pi)$ and let us suppose that the infinitesimal generator of this process $\mathcal L:D(\mathcal L)\subset L^2(\pi) \to L^2(\pi)$ is a negative operator.

 Let $V \in L^2(\pi)$ be a mean zero function on the state space of the process. 
The central limit theorem proved by Kipnis and Varadhan in \cite{KV} for 
$$
X_s=\int_0^t V(Y_s)ds\;,
$$
implies that if $V$ is in the range of  $(-\mathcal L)^{-\frac{1}{2}}$, then the limiting variance $\lim_{t \to \infty}\frac{1}{t}E[X_t^2]$ exists and is equal to
\begin{equation}
\label{H_{-1}var1}
2 \langle V\;,(-\mathcal L)^{-1}V\;\rangle_{\pi}\;.
\end{equation}

By standard arguments we can extend $\sigma^2(V,\pi)$  to a symmetric bilinear form  $\sigma^2(V,Z,\pi)$ for  $V$ and $Z$ in the range of  $(-\mathcal L)^{-\frac{1}{2}}$.  This bilinear form represents limiting covariances, and an analog to the expression (\ref{H_{-1}var1}) can be easily obtained. 

On the other hand, limiting variances and covariances can be viewed as norms in Sobolev spaces which are defined in the following lines. Properties of this spaces will be also used in Section \ref{Boltzmann-Gibbs}.

Define for $f \in D(\mathcal L)\subset L^2(\pi)$,
\begin{equation}
\label{H^1norm}
||f||^2_1\;=\; \left\langle f, (-\mathcal L_N)f \right\rangle_{\pi}\;.
\end{equation}
It is easy to see that $||\cdot||_1$ is a norm in $D(\mathcal L)$ that satisfies the parallelogram rule, and therefore, that can be extended to an inner product in $D(\mathcal L)$. We denote by $\mathscr H_1$ the completion of $D(\mathcal L)$ under the norm $||\cdot||_1$, and by $\left\langle \;,\; \right\rangle_1$ the induced inner product.
Now define
\begin{equation}
\label{variational-1}
||f||^2_{-1}\;=\; \sup_{g \in D(\mathcal L)}\left\{ 2\left\langle f, g \right\rangle_{\pi} - \left\langle g, g \right\rangle_1 \right\}\;,
\end{equation} 
and denote by $\mathscr H_{-1}$ the completion with respect to $||\cdot||_{-1}$ of the set of functions in $L^2(\pi)$ satisfying $||f||_{-1} < \infty$. 
Later we state some well known properties of these spaces.
\begin{lemma}
\label{H_1H_{-1}norms}
For $f \in L^2(\pi)\cap \mathscr H_{1} $ and $g \in L^2(\pi)\cap \mathscr H_{-1} $, we have
\begin{itemize}
\item[\textit{i})]$||g||_{-1}\;=\;\sup_{h \in D(\mathcal L) \backslash \{0\}}\frac{\left\langle h, g \right\rangle_{\pi}}{||h||_{1}}$\;,
\item[\textit{ii})]$|\left\langle f, g \right\rangle_{\pi}|\; \leq \;||f||_{1}||g||_{-1}$\;.
\end{itemize}
\end{lemma}
Property $\textit{i})$  implies that $\mathscr H_{-1}$ is the topological dual of $\mathscr H_{1}$ with respect to $L^2(\pi)$, and property $\textit{ii})$ entails that the inner product $\left\langle \;,\;  \right\rangle_{\pi}$ can be extended to a continuous bilinear form on $\mathscr H_{-1} \times \mathscr H_{1}$. The preceding results remain in force when $L^2(\pi)$ is replaced by any Hilbert space.

Observe that we can express  the central limit theorem variance in terms of the norms defined above. Indeed, $\sigma^2(V,\pi)$ is equal to
\begin{equation}
\label{H_{-1}var}
2||V||_{-1} \ \  \textit{if} \ \ V \ \   \textit{is in the range of} \ \  (-\mathcal L)^{-\frac{1}{2}}\;,
\end{equation}
or
\begin{equation}
\label{dirichletvar}
2||U||_{1}  \ \  \textit{if} \ \ V=\mathcal L U \ \  \textit{for some} \ \  U \in D(\mathcal L)\;.
\end{equation}

Now we proceed to see how to take advantage of the preceding general results in our context.  Let $L_N$ be the generator defined by
\begin{align*}
L_N(f)&=\frac{1}{2}\sum_{-N\leq x\leq x+1\leq N}X_{x,x+1}[a(p_x,p_{x+1})X_{x,x+1}(f)]\;.
\end{align*}
Note that the sum is no longer periodic.

Let $\mu_{N,y}$ be the uniform measure on the sphere 
$$
\left\{(p_{-N},\cdots,p_N) \in \mathbb R^{2N+1}\;: \sum_{i=-N}^N p_i^2=(2N+1)y^2\right\}\;,
$$
and $D_{N,y}$ the Dirichlet form associated to this measure, which is given by 
\begin{align*}
D_{N,y}(f)&=\frac{1}{2}\sum_{-N\leq x\leq x+1\leq N}\int a(p_x,p_{x+1})[X_{x,x+1}(f)]^2\mu_{N,y}(dp)\;.
\end{align*}
Is not difficult to see that the measures $\mu_{N,y}$ are ergodic for the process with generator $L_N$. 
We are interested in the asymptotic behavior, as $N$ goes to infinity, of the variance
\begin{equation}
\label{thevariance}
\sigma^2(B_N+\widehat{a}(y)A_N-H_N^F\;,\;\mu_{N,y})\;,
\end{equation}
where,
\begin{align*}
A_N(p_{-N},\cdots,p_{N})&=
p^2_{N}-p^2_{-N}\;,
\\
B_N(p_{-N},\cdots,p_{N})&=\sum_{-N\leq x\leq x+1\leq N}W_{x,x+1}\;,\\
H^F_N(p_{-N},\cdots,p_{N})&={\sum_{-N\leq {x-k}\leq {x+k}\leq N}
L_N(\tau^xF)}\;,
\end{align*}
with $F(x_{-k},\cdots,x_{k})$ a smooth function of $2k+1$ variables. 
Observe that these three classes of functions are sums of translations of local functions, and have mean zero with respect to every $\mu_{N,y}$. We introduce $\Delta_{N,y}$ to denote these variances and covariances, for instance, $\Delta_{N,y}(B_N , B_N) =\sigma^2(B_N\;,\mu_{N,y})$  and $\Delta_{N,y}(A_N , H_N^F) =\sigma^2(A_N , H_N^F\;,\;\mu_{N,y})$. The inner product in $L^2(\mu_{N,y})$ is denoted by $\langle \;,\;\rangle_{N,y}$.

Observe that the functions $B_N$ and $H_N^F$ belong to the range of $L_N$, in fact
\begin{align}
\label{B_N}
B_N(p_{-N},\cdots,p_{N})&=L_N(\sum^{N}_{x=-N}xp^2_x)\;,
\\
\label{H_N}
H^F_N(p_{-N},\cdots,p_{N})&=L_N(\psi^F_N)\;,
\end{align}
where,
$$
\psi^F_N\;=\;\sum_{-N\leq x-k\leq x+k\leq N}\tau^xF\;.
$$
This in particular  implies that the central limit theorem variances and covariances involving $B_N$ and $H_N^F$ exist. After (\ref{dirichletvar}) they are also easily computable, which is not the case for $A_N$. 

The first difficulty appearing in adapting the nongradient method to our case is to find a spherical version of telescopic sums. Such a  spherical version is obtained as a consequence of Lemma \ref{divtheorem} stated below.  We also state Lemma \ref{intoverspheres}, which provides a way to evaluate some integrals over spheres. The proofs of these and other interesting results can be founded in \cite{Ba}.
\begin{lemma} 
\label{intoverspheres}
Given $p=(p_1,\cdots,p_n), \textbf{a}=(a_1,\cdots,a_n), a=\sum^n_{k=1}a_k$ 
and $S^{n-1}(r)=\{p \in \mathbb{R}^n:|p|=r\}$ , define
\[
E(p,\textbf{a})=\prod^n_{k=1}(x^2_k)^{a_k} \quad \textit{and} \quad S_n(\textbf{a},r)=\int_{S^{n-1}(r)}E(p,\textbf{a})\ d\sigma_{n-1}
\]
then,
\[
S_n(\textbf{a},r)=\frac{2\prod^n_{k=1}\Gamma(a_k+\frac{1}{2})}{\Gamma(a+\frac{n}{2})}r^{2a+n-1}\;.
\]
Where $d\sigma_{n-1}$ denotes $(n-1)$-dimensional surface measure and $\Gamma$ is  gamma function.
\end{lemma}

\begin{corollary}
\label{duality}
There exist a constant $C$ depending on $y$ and the lower bound of $a(\cdot,\cdot)$ such that, for every $u \in D(L)$
\begin{equation*}
\left|\left\langle u,A_N \right\rangle_{N,y}\right| \leq C(2N)^{\frac{1}{2}}{D_{N,y}(u)}^{\frac{1}{2}}\;.
\end{equation*}
\end{corollary}
\begin{proof}
Observe that 
\begin{equation*}
A_N(p_{-N},\cdots,p_{N})
=\sum^{N-1}_{x=-N}X_{x,x+1}(p_xp_{x+1})\;.
\end{equation*}
Integrating by parts and applying Schwarz inequality we obtain
\begin{align*}
\left|\left\langle u,A_N \right\rangle_{N,y}\right|
                                 &=\left|\sum^{N-1}_{x=-N}\int X_{x,x+1}(u)p_xp_{x+1}\mu_{N,y}(dp)\right|\\
                                 &\leq \int \!\! \left(\sum^{N-1}_{x=-N}a(p_x,p_{x+1})(X_{x,x+1}(u))^2\right)^{\frac{1}{2}}\!\! \left(\sum^{N-1}_{x=-N}\frac{|p_xp_{x+1}|^2}{a(p_x,p_{x+1})}\right)^{\frac{1}{2}} \!\! \!\! \mu_{N,y}(dp)\;,                  
\end{align*}
which implies the desired result.
\end{proof}
As a consequence of Corollary \ref{duality} we have that the central limit theorem variances and covariances involving $A_N$  exist. In spite of that, the core of the problem will be to deal with the variance of $A_N$ which is not easily computable.

\begin{corollary}
\label{spheres1}
\begin{equation*}
\int \sum^{N-1}_{i=-N} p^2_ip^2_{i+1}\mu_{N,y}(dp)=\frac{2N(2N+1)^2}{(2N+3)(2N+1)}y^4\;.
\end{equation*}
\end{corollary}

\begin{lemma}\textit{(Divergence Theorem)}
\label{divtheorem}
Let $B^n(r)=\{p \in \mathbb {R}^n:|p| \leq r\}$ and $S^{n-1}(r)=\{p \in \mathbb {R}^n:|p| = r\}$. Then for every continuously differentiable function $f:\mathbb {R}^n \to \mathbb {R}$ we have,
\[
r\int_{B^n(r)}\frac{\partial f}{\partial p_i}(p)dp=\int_{S^{n-1}(r)}f(s_1,\cdots,s_n)s_i \ d\sigma_{n -1}.
\]
\end{lemma}
\begin{corollary}
\label{telescopicorollary}
Taking $r=y\sqrt{2N+1}$ in Lemma \ref{divtheorem}, we have for $-N\leq i,j\leq N$
$$
\int X_{i,j}(f)p_ip_{j}\mu_{N,y}(dp) =
\frac{r}{|S^{2N}(r)|} \int_{B^{2N+1}(r)}\left(p_i\frac{\partial f}{\partial p_i}-p_j\frac{\partial f}{\partial p_j}\right)dp\;,
$$
where $X_{i,j}=p_j\frac{\partial}{\partial p_i}-p_i\frac{\partial}{\partial p_j}$.
\end{corollary}

Corollary \ref{telescopicorollary} is extremely useful for us, because it provides a way to perform telescopic sums over the sphere. In fact, it implies that given $-N \leq i < j \leq N$ we have
\begin{equation}
\label{telescopic}
\int X_{i,j}(f)p_ip_{j}\mu_{N,y}(dp)= \int \sum_{k=i}^{j-1} X_{k,k+1}(f)p_kp_{k+1}\mu_{N,y}(dp)\;.
\end{equation}
We should stress the fact that equality of the integrands is false, which is not the case in the planar setting.

Now we return to the study of $A_N$,$B_N$ and $H_N^F$. The next proposition entails the asymptotic behavior, as $N$ goes to infinity, of the central limit theorem variances and covariances involving $B_N$ or $H_N^F$, besides an estimate in the case of $A_N$.
\begin{theorem}
\label{thecovariances}
The following limits hold locally uniformly in $y > 0$.
\begin{align*}
\textit{i})\quad\quad&\lim_{N\to \infty}\frac{1}{2N}\Delta_{N,y}(H^F_N,H^F_N)\;=\;\textbf{E}_{\nu_y}[a(p_0,p_1)[X_{x,x+1}(\widetilde{F})]^2]\;,\\
\textit{ii})\quad\quad&\lim_{N\to \infty}\frac{1}{2N}\Delta_{N,y}(B_N,B_N)\;=\;\textbf{E}_{\nu_y}[a(p_0,p_1)(2p_0p_1)^2]\;,\\
\textit{iii})\quad\quad&\lim_{N\to \infty}\frac{1}{2N}\Delta_{N,y}(B_N,H^F_N)\;=\;-\textbf{E}_{\nu_y}[2p_0p_1a(p_0,p_1)X_{x,x+1}(\widetilde{F})]\;,\\
\textit{iv})\quad\quad&\lim_{N\to \infty}\frac{1}{2N}\Delta_{N,y}(A_N,B_N)\;=\;-\textbf{E}_{\nu_y}[(2p_0p_1)^2]\;,\\
\textit{v})\quad\quad&\lim_{N\to \infty}\frac{1}{2N}\Delta_{N,y}(A_N,H^F_N)\;=\; 0\;,\\
\textit{vi})\quad\quad&\limsup_{N\to \infty}\frac{1}{2N}\Delta_{N,y}(A_N,A_N)\;\leq\; C\;,
\end{align*}
where $\widetilde{F}$ is formally defined by
$$
\widetilde{F}(p)=\sum_{j=-\infty}^{\infty}\tau^jF(p)
$$
and $C$ is a positive constant depending uniformly on $y$. Although $\widetilde{F}$ do not really make sense, the gradients $X_{i,i+1}(\widetilde{F})$ are all well defined.
\end{theorem}

\begin{proof}
\textit{i})
From (\ref{H_N}) and (\ref{dirichletvar}) we have that
\begin{align*}
\Delta_{N,y}(H^F_N,H^F_N)&=2D_{N,y}(\psi^F_N)\\
                         &=\int\sum^{N-1}_{x=-N}a(p_x,p_{x+1})[X_{x,x+1}(\psi^F_N)]^2 \mu_{N,y}(dp)\;.
\end{align*}
The sum in the last line can be broken into two sums, the first one considering the indexes in $\{-N+2k,\cdots,N-2k-1\}$ and the second one considering the indexes in the complement with respect to $\{-N,\cdots,N-1\}$. 
From the conditions imposed over $F$, when divided by $N$, the term corresponding to the second sum tends to zero as $N$ goes to infinity. Then,
\begin{align*}
\lim_{N\to \infty}\frac{1}{2N}\Delta_{N,y}(H^F_N,H^F_N)&=
\lim_{N\to \infty}\int\frac{1}{2N}\sum^{N-1}_{x=-N}a(p_x,p_{x+1})[X_{x,x+1}(\psi^F_N)]^2 \mu_{N,y}(dp)\\
&=\lim_{N\to \infty}\int\frac{1}{2N}\sum^{N-1}_{x=-N}\tau^xg(p) \mu_{N,y}(dp)\;,
\end{align*}
where $g(p)=a(p_0,p_1)[X_{x,x+1}(\widetilde{F})]^2$ and 
\[
\widetilde{F}(p)=\sum_{j=-\infty}^{\infty}\tau^jF(p).
\]

The desired result comes from the rotation invariance of $\mu_{N,y}$ together with the equivalence of ensembles stated in Appendix \ref{equiensem}.

\textit{ii}) 
This is proved in the same way as \textit{i}) by using property (\ref{B_N}).

\textit{iii})
This is proved in the same way as \textit{i}) by using property (\ref{H_N}) and the fact that $W_{x,x+1}=-X_{x,x+1}[p_xp_{x+1}a(p_x,p_{x+1})]$.

\textit{iv}) 
This is proved in the same way as \textit{i}) by using property (\ref{B_N}), the fact that $A_N = X_{N,-N}(p_Np_{-N})$ and Corollary \ref{spheres1}.

\textit{v})
This is proved by the same arguments used in the preceding items together with the telescopic sum obtained in (\ref{telescopic}).

\textit{vi}) 
By duality ({\em {c.f}} Lemma \ref{H_1H_{-1}norms}) $\Delta_{N,y}(A_N,A_N)= 2c^2 $, where $c$ is the smallest constant such that for every $u \in D(L)$,
\begin{equation}
\label{duality2}
\left|\left\langle u,A_N \right\rangle_{N,y}\right| \leq c{D_{N,y}(u)}^{\frac{1}{2}}\;.
\end{equation} 
Recall that Corollary \ref{duality} ensures the existence of a constant $C$ depending locally uniformly on $y$, such that  $C(2N)^{\frac{1}{2}}$ satisfies (\ref{duality2}) for every $u \in D(L)$. Therefore, $c$ is smaller than $C(2N)^{\frac{1}{2}}$ and
$$
\frac{1}{2N}\Delta_{N,y}(A_N,A_N) \leq 
2 C^2 \;,
$$
which concludes the proof of Theorem \ref{thecovariances}.
\end{proof}
We proceed now to calculate the only missing limit variance (the one corresponding to $A_N$) in an indirect way, as follows.

Using the basic inequality 
$$
|\Delta_{N,y}(A_N,B_N - H_N^F)|^2 \leq \Delta_{N,y}(A_N,A_N)\Delta_{N,y}(B_N - H_N^F,B_N - H_N^F)\;,
$$ 
together with Theorem \ref{thecovariances}, we obtain
\begin{equation}
\label{liminfA_N1}
\frac{(4 y^4)^2}{\textbf{E}_{\nu_y}[a(p_0,p_1)\{2p_0p_1+X_{0,1}(\widetilde{F})\}^2]}
\leq 
\liminf_{N \to \infty}\frac{\Delta_{N,y}(A_N,A_N)}{2N}\;.
\end{equation}
Let us define $\hat{a}(y)$ by the relation
\begin{equation}
\label{^a1}
\hat{a}(y)=y^{-4}\inf_{F} a(y,F)\;,
\end{equation}
where the infimum is taken over all local smooth functions, and
\begin{equation}
\label{^a2}
a(y,F)=\textbf{E}_{\nu_y}[a(p_0,p_1)(p_0p_1+X_{0,1}(\widetilde{F}))^2]\;.
\end{equation}
Since the limit appearing in (\ref{liminfA_N1}) does not depend of $F$, we have 
\begin{equation}
\label{lowerbound}
\liminf_{N \to \infty}\frac{\Delta_{N,y}(A_N,A_N)}{2N}\geq 
\frac{4 y^4}{\widehat{a}(y)}\;.
\end{equation}
Moreover, this limit is locally uniform in $y$.
 
We are now ready to state the main result of this section.
\begin{theorem}
\label{covarA_N}
The function $\widehat{a}(y)$ is continuous in $y>0$ and  
\begin{equation}
 \label{covarianceA_N}
\lim_{N\to \infty}\frac{1}{2N}\Delta_{N,y}(A_N,A_N)=\frac{4y^4}{\hat{a}(y)}\;.
\end{equation}
\end{theorem}
\begin{proof}
Let us define
\begin{equation}
l(y)=\limsup_{\substack{N \to \infty \\ y' \to y}}\frac{1}{2N}\Delta_{N,y'}(A_N,A_N)\;.
\end{equation} 
By definition, $\widehat{a}(y)$ and $l(y)$ are upper semicontinuous functions. 

In order to prove (\ref{covarianceA_N}) it is enough to verify the following equality
\begin{equation}
\label{equalityl(y)}
l(y) = \frac{4y^4}{\hat{a}(y)}.
\end{equation}
In fact, from the definition of $l(y)$, it is clear that
$$
\limsup_{\substack{N \to \infty}}\frac{1}{2N}\Delta_{N,y}(A_N,A_N) \leq l(y)\;,
$$
which together with (\ref{lowerbound}) proves (\ref{covarianceA_N}). Moreover, equality (\ref{equalityl(y)}) together with the upper semicontinuity of $l(y)$ gives the lower semicontinuity of $\widehat{a}(y)$. Ending the proof of Theorem \ref{covarA_N}.  

On the other hand, from the upper semicontinuity of $\hat{a}(y)$ together with the lower bound in (\ref{lowerbound}) we obtain $\frac{4y^4}{\hat{a}(y)}\leq l(y)$. Therefore, it remains to check the validity of the opposite inequality, which is equivalent to prove that for  every $\theta \in \mathbb R$ 
\begin{equation}
\label{2inequalityl(y)}
l(y) > \theta \quad \quad \Longrightarrow  \quad \quad \theta \leq \frac{4y^4}{\hat{a}(y)}\;.
\end{equation}
Suppose that $l(y) > \theta$. Then, there exist a sequence $y_N \to y$ such that
$$
\lim_{N \to \infty}\frac{1}{2N} \Delta_{N,y_N}(A_N,A_N)=A>\theta\;.
$$
By $\textit{i})$ in Lemma \ref{H_1H_{-1}norms} we have
$$\{\Delta_{N,y_N}(A_N,A_N)\}^{1/2}\;=\;\sup_{h \in D(L_N) \backslash \{0\}}\frac{\left\langle h, A_N \right\rangle_{N,y_N}}{{D_{N,y}(h)}^{\frac{1}{2}}}\;.
$$
Then, there exist a sequence of smooth functions $\{w_N \}_{N \geq 1}$ such that $w_N\in Dom(L_N)$ and 
$$
{\langle w_N,A_N \rangle}_{N,y_N} >\sqrt{N\theta} {{D}_{N,y_N}(w_N)}^\frac{1}{2}\;.
$$
We can suppose without loss of generality that
\begin{equation}
\label{microzeromean}
{\langle w_N \rangle}_{N,y_N}=0\;.
\end{equation}
We can renormalize by taking $\gamma_N=\frac{1}{y_N^2}\frac{1}{2N}{\langle w_N,A_N\rangle}_{N,y_N}$ and $u_N=\gamma_N^{-1}w_N$ , obtaining a sequence of functions $\{u_N \}_{N \geq 1}$ satisfying
\begin{equation}
\label{microcondition1}
\frac{1}{2N}\int X_{N,-N}(u_N)p_Np_{-N}\mu_{N,y_N}(dp)= y_N^2\;,
\end{equation}
and
\begin{equation}
\label{microcondition2}
\limsup_{N \to \infty}\frac{1}{2N}\int\sum_{-N\leq x\leq x+1\leq N} a(p_x,p_{x+1})[X_{x,x+1}(u_N)]^2\mu_{N,y_N}(dp)\leq \frac{4 y^4}{\theta}\;.
\end{equation}

The aim of Lemma \ref{lemmakblock}, Lemma \ref{hat{a}lemma2} and Lemma \ref{hat{a}lemma3} proved below, is to use the sequence $\{u_N \}_{N \geq 1}$, together with properties (\ref{microcondition1}) and (\ref{microcondition2}), in order to obtain a function $\xi$ satisfying
\begin{align*}
\textit{i})\quad\quad& \textbf{E}_{\nu_y}[\xi]=0\;,\\
\textit{ii})\quad\quad& \textbf{E}_{\nu_y}[p_0p_1\xi]= 0\;,\\
\textit{iii})\quad\quad& \textbf{E}_{\nu_y}[a(p_0,p_1)\{p_0p_1 + \xi\}^2] \leq \frac{4 y^4}{\theta}\;,
\end{align*}
besides an additional condition concerning $X_{i,i+1}\tau^j\xi - X_{j,j+1}\tau^i\xi$ (see Lemma \ref{hat{a}lemma3}).

Condition \textit{iii}) obviously implies
\begin{equation}
\label{boundtheta}
\theta \leq \frac{4 y^4}{\textbf{E}_{\nu_y}[a(p_0,p_1)\{p_0p_1 + \xi\}^2]}\;.
\end{equation}
Rather less obvious is the fact that \textit{i}), \textit{ii}) and the extra condition on $X_{i,i+1}\tau^j\xi - X_{j,j+1}\tau^i\xi$,  imply that $\xi$ belongs to the closure in $L^2(\nu_y)$ of the set over which the infimum in the definition of $\widehat{a}(y)$ is taken (see (\ref{^a1})). The proof of this fact is the content of Section \ref{{H}_y}.

In short, supposing $l(y) > \theta$ we will find a function $\xi$ such that $\textbf{E}_{\nu_y}[a(p_0,p_1)\{p_0p_1 + \xi\}^2] \leq \frac{4 y^4}{\theta}$.  Additionally we will see that such a function belongs to the closure of $\{X_{0,1}(\widetilde{F}): F \; \text{is a local smooth function}\}$.  These two facts imply the left hand side of (\ref{2inequalityl(y)}), finishing the proof of Theorem \ref{covarA_N}.
\end{proof}
Now we state and prove the lemmas concerning the construction of the function $\xi$ endowed with the required properties.
\begin{lemma}
\label{lemmakblock}
Given $\theta > 0$, $k \in \mathbb N$ and a convergent sequence of positive real numbers  $\{y_N\}_{N \geq 1}$ satisfying (\ref{microcondition1}) and (\ref{microcondition2}), there exists a sequence of functions $\{u_N^{(k)}\}_{N \geq 1}$ depending on the variables $p_{-k}, \cdots,p_{k}$ such that
\begin{equation}
\label{microcondition3}
\frac{1}{2k}\int X_{k,-k}(u_N^{(k)})p_kp_{-k}\mu_{N,y_N}(dp)= y_N^2\;,
\end{equation}
and,
\begin{equation}
\label{microcondition4}
\limsup_{N \to \infty}\frac{1}{2k}\int\sum_{-k\leq x\leq x+1\leq k} a(p_x,p_{x+1})[X_{x,x+1}(u_N^{(k)})]^2\mu_{N,y_N}(dp)\leq \frac{4 y^4}{\theta}\;,
\end{equation}
where $y$ is the limit of $\{y_N\}_{N \geq 1}$.
\end{lemma}
\begin{proof}
Define for $-N\leq x\leq x+1\leq N$
\[
\alpha^N_{x,x+1}=y_N^{-2}\textbf{E}_{\mu_{N,y_N}}[p_xp_{x+1}X_{x,x+1}(u_N)]
\]
and
\[
\beta^N_{x,x+1}=y_N^{-4}\textbf{E}_{\mu_{N,y_N}}[a(p_x,p_{x+1})[X_{x,x+1}(u_N)]^2]\;.
\]
Where $\textbf{E}_{\mu_{N,y_N}}$ denotes integration with respect to $\mu_{N,y_N}$ and $u_N$ is the function satisfying (\ref{microcondition1}) and (\ref{microcondition2}).

Thanks to (\ref{microcondition1})  and the telescopic sum obtained in (\ref{telescopic}) we have 
\[
\alpha^N_{-N,-N+1}+\cdots +\alpha^N_{N-1,N}=2N\;.
\]
After (\ref{microcondition2}) for every $\epsilon >0$ there exist $N_0$ such that
\[
\beta^N_{-N,-N+1}+\cdots +\beta^N_{N-1,N}\leq 2N(\frac{4}{\theta}+\epsilon y_N^{-4})\;,
\]
for all $N\geq N_0$.

By using Lemma \ref{lemaux} stated and proved below, we can conclude the existence of a block $\Lambda_{N,k}$ of size $2k$ contained in $\{-N,\cdots,N\}$ such that
\begin{equation}
\label{alphabetha}
\gamma_N \left( \sum_{x,x+1 \in \Lambda_{N,k}}\alpha^N_{x,x+1} \right)^2  \geq 2k \sum_{x,x+1 \in \Lambda_{N,k}}\beta^N_{x,x+1}\;, 
\end{equation}
with $\gamma_N = \frac{4}{\theta}+\epsilon y_N^{-4}$.

Let us now introduce some notation. Denote by $R_N$ the rotation of axes defined as
$$
\begin{array}{cccc}
R_N :& \mathbb R^{2N+1} &\to& \mathbb R^{2N+1} \\
&\quad (p_{-N}, \cdots, p_{N}) &\to& (p_{-N+1}, \cdots, p_{N},p_{-N})\;.
\end{array}
$$
For an integer $i > 0$ we denote by $R_N^i$ the composition of $R_N$ with itself $i$ times, for $i<0$ the inverse of $R_N^{-i}$ by $R_N^i$ , and $R_N^0$ for the identity function. As usually, given a function $u:\mathbb R^{2N+1} \to \mathbb R\;$ we define $R_N^iu=u \circ R_N^i$.

Let us define 
$$w_N = R_N^i u_N\;,$$
where $i$ is an integer satisfying
$\Lambda_{N,k} = \{{i-k}, \cdots ,{i+k} \}\;.$

Now we proceed to check that (\ref{microcondition3}) and (\ref{microcondition4}) are satisfied by the sequence of functions $\{u_N^{k}\}_{N \geq 1}$ defined as 
$$
u_N^{k}=2k \left( \sum_{x,x+1 \in \Lambda_{N,k}}\alpha^N_{x,x+1}\right)^{-1}\textbf{E}_{\mu_{N,y_N}}[w_N \;|\; \Lambda_k]\;,
$$
where $\Lambda_k =\{p_{-k},\cdots,p_{k}\}$.
Because of the invariance under axes rotation of the measure $\mu_{N,y_N}$, together with the relation
$$
X_{x,x+1}(w_N)=R^{-i} X_{x+i,x+i+1}(u_N)\;,
$$ 
we have that  $\textbf{E}_{\mu_{N,y_N}}[p_xp_{x+1}X_{x,x+1}(u^k_N)]$ is equal to
$$
2k \left( \sum_{x,x+1 \in \Lambda_{N,k}}\alpha^N_{x,x+1}\right)^{-1} \!\! \!\! \!\!
\textbf{E}_{\mu_{N,y_N}}[p_{x+i}p_{x+i+1}X_{x+i,x+i+1}(u_N)]
$$
for all $x $ such that $ \{p_x,p_{x+1}\} \subset \Lambda_k$. Then, summing over $x $ we obtain that the left hand side of (\ref{microcondition3}) is equal to
$$
\left( \sum_{x,x+1 \in \Lambda_{N,k}}\alpha^N_{x,x+1}\right)^{-1}
\textbf{E}_{\mu_{N,y_N}}[\sum_{x,x+1 \in \Lambda_{N,k}}p_{x}p_{x+1}X_{x,x+1}(u_N)]\;,
$$ 
which in turns is equal to $y^2_N$, proving (\ref{microcondition3}).

Using Jensen's inequality, and an analogous argument as the one used in the preceding lines, we obtain that
$$
\textbf{E}_{\mu_{N,y_N}}[a(p_{x},p_{x+1})[X_{x,x+1}(u_N^k)]^2]\;
$$
is bounded above by
$$
4k^2\left(\sum_{x,x+1 \in \Lambda_{N,k}}\alpha^N_{x,x+1}\right)^{-2}
\textbf{E}_{\mu_{N,y_N}}[a(p_{x+i},p_{x+i+1})[X_{x+i,x+i+1}(u_N)]^2]\;,
$$ 
for all $x $ such that $ \{p_x,p_{x+1}\} \subset \Lambda_k$. This implies (\ref{microcondition4}) after adding over $x$, using relation (\ref{alphabetha}) and taking the superior limit as $N$ goes to infinity.
\end{proof}
Now we state and proof the technical result used to derive $\ref{alphabetha}$.
\begin{lemma}
\label{lemaux}
Let $\{a_i\}_{i=1}^{m}$ and $\{b_i\}_{i=1}^{m}$ two sequences of real and positive real numbers, respectively, satisfying 
\begin{equation}
\label{hyplemaux}
\sum_{i=1}^{m}a_i = m  \quad \quad \textit{and} \quad \quad \sum_{i=1}^{m}b_i \leq m\gamma\;,
\end{equation}
for fixed constants $m \in \mathbb N$, $\gamma > 0$ and $k << m$. Then, there exists a block $\Lambda$ of size $2k$ contained in the discrete torus $\{1,\cdots,m\}$ such that
\begin{equation}
\label{teslemaux}
\gamma \left( \sum_{i \in \Lambda}a_i \right)^2  \geq 2k \sum_{i \in \Lambda}b_i\;. 
\end{equation}
\end{lemma}
\begin{proof}
It is enough to check the case where $2k$ is a factor of $m$. In fact, in the opposite case we can consider periodic sequences of size $2km$ instead of the originals ones $\{a_i\}_{i=1}^{m}$ and $\{b_i\}_{i=1}^{m}$.

Therefore we can suppose that $m=2kl$ for some integer $l$, and define for $i \in \{1,\cdots,l\}$
$$
\alpha_i = \sum_{x \in \Lambda_i}a_x \quad \quad \textit{and} \quad \quad 
\beta_i = \sum_{x \in \Lambda_i}b_x\;,
$$
where $\Lambda_{i} = \{2k(i-1), \cdots, 2ki\}$. 

We want to conclude that (\ref{teslemaux}) is valid for at least one of the $\Lambda_{i}$'s. Let us argue by contradiction.

Suppose that $\sqrt{2k\beta_i}>\alpha_i \gamma^{\frac{1}{2}}$ for every
$i=1,\cdots,l$. Adding over $i$ and using the first part of hypothesis (\ref{hyplemaux}), we obtain
$$
\sum^l_{i=1}\sqrt{2k\beta_i}> m{\gamma}^{\frac{1}{2}}\;.
$$
By squaring both sides of the last inequality we have,
$$
\sum^l_{i=1}\beta_i>m\gamma \;,
$$
which is in contradiction with the second part of hypothesis (\ref{hyplemaux}).
\end{proof}
Now we proceed to take, for each positive integer $k$, a weak limit of the sequence $\{u_N^{(k)}\}_{N \geq 1}$ obtained in Lemma \ref{lemmakblock}.
\begin{lemma}
\label{hat{a}lemma2}
For each positive integer $k$ there exists a function $\widetilde{u}_k$ depending on the variables $p_{-k}, \cdots,p_{k}$ such that
\begin{equation}
\label{macrocondition1}
\frac{1}{2k}\textbf{E}_{\nu_y}[X_{k,-k}(\widetilde{u}_k)p_kp_{-k}]= y^2\;,
\end{equation}
and
\begin{equation}
\label{macrocondition2}
\frac{1}{2k}\textbf{E}_{\nu_y}\left[\sum_{-k\leq x\leq x+1\leq k} a(p_x,p_{x+1})[X_{x,x+1}(\widetilde{u}_k)]^2\right] \leq \frac{4 y^4}{\theta}\;.
\end{equation}
\end{lemma}
\begin{proof}
Consider the linear functionals $\Lambda_{i,i+1}^N$ defined for $-k\leq i\leq i+1\leq k$ by
$$
\begin{array}{cccc}
\Lambda_{i,i+1}^N :&  L^2(\mathbb{R}^{2k+1};\nu_y) &\to& \mathbb R \\
&\quad w &\to& \textbf{E}_{\mu_{N,y_N}}[X_{i,i+1}(u_N^{(k)})w]\;.
\end{array}
$$
Let $\mathcal{P}^k$ be an enumerable dense set of polynomials in $L^2(\mathbb{R}^{2k+1};\nu_y)$. From (\ref{macrocondition2}) and the Cauchy-Schwartz inequality  we obtain the existence of a constant $C$ such that 
\[
|\Lambda_{i,i+1}^N(w)|\leq C \left( \int w^2 d\mu_{N,y_N} \right)^{\frac{1}{2}}\;,
\]
for every $w \in \mathcal{P}^k$.

By a diagonal argument we can draw a subsequence for which the limits of $\Lambda_{i,i+1}^N(w)$ exist for all $w \in \mathcal{P}^k$. Moreover, passing to the limit and extending to the whole space $L^2(\mathbb{R}^{2k+1};\nu_y)$, we get linear functionals $\Lambda_{i,i+1}$ satisfying
\begin{equation}
\label{linearfunc1}
|\Lambda_{i,i+1}(w)|\leq C \left( \int w^2 d\nu_y \right)^{\frac{1}{2}}.
\end{equation}
On the other hand, consider the linear functionals $\Lambda^N$ defined by
$$
\begin{array}{cccc}
\Lambda^N :&  L^2(\mathbb{R}^{2k+1};\nu_y) &\to& \mathbb R \\
&\quad w &\to& \textbf{E}_{\mu_{N,y_N}}[u_N^{(k)}w]\;.
\end{array}
$$
Because of (\ref{macrocondition2}), (\ref{microzeromean}) and Poincare's inequality we have
$$
\textbf{E}_{\mu_{N,y_N}}[(u_N^{(k)})^2] 
\leq
\frac{4 y^4C}{\theta}\;,
$$
for a constant $C$ depending only on $k$. Then, by the very same arguments used above, we get a linear functional $\Lambda$ satisfying
\begin{equation}
\label{linearfunc2}
|\Lambda(w)|\leq \sqrt C \left( \int w^2 d\nu_y \right)^{\frac{1}{2}}.
\end{equation}
Finally, it follows from (\ref{linearfunc1}) and (\ref{linearfunc2}) the existence of a function $\widetilde{u}_k$ satisfying
\[
\Lambda_{i,i+1}(w)=\textbf{E}_{\nu_y}[X_{i,i+1}(\widetilde{u}_k)w] \;,
\]
\[
\Lambda(w)=\textbf{E}_{\nu_y}[\widetilde{u}_k w] \;,
\]
and therefore, satisfying (\ref{macrocondition1}) and (\ref{macrocondition2}).
\end{proof}
Now using the sequence $\{\widetilde{u}_{k}\}_{k\in \mathbb N}$ we construct a sequence of functions $\{u_{k'}\}_{k'\in M}$ indexed on an infinite subset of $\mathbb N$, each one depending on the variables $p_{-k'}, \cdots,p_{k'}$. This sequence will satisfy, besides (\ref{macrocondition1}) and (\ref{macrocondition2}), an additional condition regarding the contribution of the terms near the boundary of $\{-k',\cdots,k'\}$ to the total Dirichlet form. 
\begin{lemma}
\label{dirichboundterms}
There exist a sequence of functions $\{u_{k'}\}_{k'\in M}$ indexed on an infinite subset of $\mathbb N$, each one depending on the variables $p_{-k'}, \cdots,p_{k'}$, satisfying  (\ref{macrocondition1}), (\ref{macrocondition2}) and
$$
\textbf{E}_{\nu_y}[(X_{x,x+1}(u_{k'}))^2]=O(k^{7/8})\;,
$$  
for $\{x,x+1\} \subset I_{k'} \cup J_{k'}$, where $ I_{k'}=[-k',-k'+(k')^{1/8}]$ and $J_{k'}=[k'-(k')^{1/8},k']$.
\end{lemma}    
\begin{proof}
Given $k>0$ divide each interval $[-k,-k+k^{1/4}]$ and $[k-k^{1/4},k]$  into $k^{1/8}$ blocks of size $k^{1/8}$, and consider the sequence $\{\widetilde{u}_{k}\}_{k\in \mathbb N}$ obtained in Lemma \ref{hat{a}lemma2}. 

Because of (\ref{macrocondition2}), for every $k>0$ there exist $k' \in [k-k^{1/8},k]$ such that 
$$
\textbf{E}_{\nu_y}\left[\sum_{\{x, x+1\} \in I_{k'} \cup J_{k'}} (X_{x,x+1}(\widetilde{u}_k))^2\right] = O(k^{7/8})\;.
$$
Define for each $k>0$ the function 
$
u_{k'}=\frac{1}{C_{y,k,k'}}\textbf{E}_{\nu_y}\left[\widetilde{u}_{k}\;|\;\mathfrak{F}_{-k'}^{k'}\right]
$
, where
$$
C_{y,k,k'}=\frac{1}{2k'y^2}\left\{2ky^2-\textbf{E}_{\nu_y}\left[\sum_{\{x, x+1\} \in I_{k'} \cup J_{k'}} p_xp_{x+1}X_{x,x+1}(\widetilde{u}_k)\right] \right\}
$$
and $\mathfrak{F}_{-k'}^{k'}$ denotes the $\sigma$-field generated by $p_{-k'},\cdots,p_{k'}$.

Is easy to see that the sequence $\{u_{k'}\}_{k'}$ satisfies the desired conditions. 
\end{proof}
Finally, we obtain the weak limit used in the proof of Theorem \ref{covarA_N}.
\begin{lemma}
\label{hat{a}lemma3}
There exist a function $\xi$ in $L^2(\nu_y)$ satisfying 
\begin{equation}
\label{final1}
\textbf{E}_{\nu_y}[\xi]=0,
\end{equation}
\begin{equation}
\label{final2}
\textbf{E}_{\nu_y}[p_0p_1\xi]=0,
\end{equation}
\begin{equation}
\label{final3}
\textbf{E}_{\nu_y}[a(p_0,p_1)[p_0p_1+\xi]^2]\leq\frac{4 y^4}{\theta},
\end{equation}
and the integrability conditions
\begin{equation}
\label{final4}
X_{i,i+1}(\tau^j\xi)=X_{j,j+1}(\tau^i\xi) \ \ \ \ \textit{if}\ \ \ \  \{i,i+1\}\cap\{j,j+1\}=\emptyset, 
\end{equation}
\begin{equation}
\label{final5}
p_{i+1}[X_{i+1,i+2}(\tau^i\xi)-X_{i,i+1}(\tau^{i+1}\xi)]=p_{i+2}\tau^i\xi-p_{i}\tau^{i+1}\xi\ \ \textit{for} \ \  i\in \mathbb{Z}.
\end{equation}
\end{lemma}
\begin{proof}
For all integer $k>0$ define  
\[
\zeta_k=\frac{1}{2{k'}}\sum_{i=-{k'}}^{{k'}-1}X_{i,i+1}(u_{k'})(\tau^{-i}).
\]
It is clear from the definition of $\zeta_k$ that 
$\textbf{E}_{\nu_y}[\zeta_k]=0$
and
$$
p_0p_1\zeta_{k'}(\omega)=\frac{1}{2{k'}}\sum_{i=-{k'}}^{{k'}-1}\{p_ip_{i+1}X_{i,i+1}(u_{k'})\}(\tau^{-i}).
$$
Therefore, after (\ref{macrocondition1}) it follows that
$\textbf{E}_{\nu_y}[p_0p_1\zeta_k]=y^4.$
Moreover, by using Schwarz inequality, translation invariance of the measure and condition (\ref{macrocondition2}), we obtain
\begin{align*}
\textbf{E}_{\nu_y}[a(p_0,p_1)\zeta^2_k]
&\leq \frac{4 y^4}{\theta}.
\end{align*}
Now consider the sequence $\{\xi_k\}_{k\geq 1}$ defined by
$$\xi_k=\zeta_k-p_0p_1.$$
Since the preceding sequence is uniformly bounded in $L^2(\nu_y)$, there exist a weak limit function $\xi \in L^2(\nu_y)$. Clearly, the function $\xi$ satisfies (\ref{final1}), (\ref{final2}) and (\ref{final3}).

In addition, an elementary calculation shows that (\ref{final4}) and (\ref{final5}) are satisfied by $\xi_k$ up to an error coming from a small number of terms near the edge of $[-k',k']$. Then, in view of Lemma \ref{dirichboundterms}, the final part of the lemma is satisfied as well.
\end{proof}

\section{Boltzmann-Gibbs Principle}
\label{Boltzmann-Gibbs}

The aim of this section is to provide a proof for Theorem \ref{BoltzGibbs}. In fact, we will prove a stronger result that will be also useful in the proof of tightness. Namely, 
{\small
\begin{equation*}
\lim_{k \to \infty} \lim_{N \to \infty}\mathbb E_{\nu^N_y}\left[\sup_{0 \leq t \leq T}
\left(\int^t_0 \sqrt{N} \sum_{x \in \mathbb T_N}\nabla_N H_s(x/N) [V_x(p(s))-\mathcal{L}_N \tau^xF_k(p(s))]\; ds\right)^2
\right]=0
\end{equation*}
}
where,
\begin{equation*}
V_x(p) = W_{x,x+1}(p)-\hat{a}(y)[p^2_{x+1}-p^2_x] \;.
\end{equation*}
We begin localizing the problem. Fix an integer $M$ that shall increase to infinity after $N$. Being $l$ and $r$ the integers satisfying $N= lM+r$ with $0 \leq r < M$, define for $j = 1, \cdots, l$
\begin{align*}
B_j &=\{(j-1)M+1,\cdots,jM\} \;, \\ 
B^{\prime}_j &=\{(j-1)M+1,\cdots,jM-1\}  \;, \\
B^{k}_j &=\{(j-1)M+1,\cdots,jM-s_k\} \;, 
\end{align*}
where $s_k$ is the size of the block supporting $F_k$. Define the remaining block as $B_{l+1} =\{lM+1,\cdots,N\}$. With this notations we can write
\begin{equation}
\label{V_1V_2V_3}
\sqrt{N} \sum_{x \in \mathbb T_N}\nabla_N H_s(x/N) [V_x(p(s))-\mathcal{L}_N \tau^xF_k(p(s))]\;=\;V_1+V_2+
V_{3}
\end{equation}
with,
\begin{align*}
V_{1}&=\sqrt{N} \sum_{j=1}^{l}\sum_{x \in B_j^{\prime}}\nabla_N H_s(x/N) V_x
-\sqrt{N}\sum_{j=1}^{l}\sum_{x \in B_j^{k} } \nabla_N H_s(x/N) \mathcal{L}_N \tau^xF_k\;,\\
V_2&=\sqrt{N} \sum_{x \in B_{l+1}}\nabla_N H_s(x/N) V_x\;
-\;\sqrt{N} \sum_{x \in B_{l+1}}\nabla_N H_s(x/N) \mathcal{L}_N \tau^xF_k\;,\\
V_3&=\sqrt{N} \sum_{j=1}^{l}\nabla_N H_s(jK/N) V_{jM} -\sqrt{N}
\sum_{j=1}^{l}\sum_{x \in B_j \backslash B_j^{k} }\nabla_N H_s(x/N) \mathcal{L}_N \tau^xF_k\;. 
\end{align*}
Observe that $V_{1}$ is a sum of functions which depends on disjoint blocks, and contains almost all the terms appearing in the left hand side of (\ref{V_1V_2V_3}), therefore, $V_{2}$  and $V_{3}$ can be considered as error terms.
In order to prove Theorem \ref{BoltzGibbs} it suffices to show
\begin{equation}
\label{Vi's}
\lim_{k \to \infty}\lim_{N \to \infty} \mathbb E_{\nu^N_y}\left[\sup_{0 \leq t \leq T}
\left(\int^t_0 V_i \; ds\right)^2
\right]\;=\;0\;,
\end{equation}
for each $V_i$ separately.

The following is a very useful estimate of the time variance in terms of the $\mathscr H_{-1}$ norm defined in (\ref{variational-1}). 

\begin{proposition}
\label{cenlimvarian}
Given $T>0$ and a mean zero function  $V \in L^2(\pi) \cap \mathscr H_{-1}$, 
$$
\mathbb E_{\pi}\left[\sup_{0 \leq t \leq T}
\left(\int^t_0 V(p_s) \; ds\right)^2
\right]\; \leq \;
24T ||V||^2_{-1}\;.
$$
\end{proposition}
See Lemma 2.4 in \cite{KmLO}  or Proposition 6.1 in \cite{KL} for a proof . 
\begin{remark}
\label{cenlimvarianremark}
A slightly modification in the proof given in \cite{KmLO} permit to conclude that, for every smooth function $h:[0,T] \to \mathbb R$, 
$$
\mathbb E_{\pi}\left[\sup_{0 \leq t \leq T}
\left(\int^t_0 h(s)V(p_s) \; ds\right)^2
\right]\; \leq \;
C_h ||V||^2_{-1}\;,
$$
where,
$$
C_h \;=\;6 \{ 4||h||^2_{\infty}T^2 + ||h^{\prime}||^2_{\infty} T^3\}\;.
$$
Moreover, in our case we have
$$
\mathbb E_{\nu_y^N}\left[\sup_{0 \leq t \leq T}
\left(\int^t_0 \sum_{j=1}^l h_j(s)V_{B_j}(p_s) \; ds\right)^2
\right]\; \leq \;
\sum_{j=1}^lC_{h_j} ||V_{B_j}||^2_{-1}\;,
$$
for functions $\{V_{B_j}\}_{j=1}^l$ depending on disjoint blocks.
\end{remark}
The proof of (\ref{Vi's}) will be divided in three lemmas.
\begin{lemma}
\begin{equation*}
\lim_{k \to \infty}\lim_{N \to \infty} \mathbb E_{\nu^N_y}\left[\sup_{0 \leq t \leq T}
\left(\int^t_0 V_2 \; ds\right)^2
\right]\;=\;0\;.
\end{equation*}
\end{lemma}
\begin{proof}
By Proposition \ref{cenlimvarian}, the expectation in the last expression is bounded above by
 the sum of the following three terms
\begin{equation}
\label{V2.1}
\frac{3C_H |B_{l+1}|}{N} \sum_{x \in B_{l+1}} \left\langle  W_{x,x+1}
\;,\;(-\mathcal L_{N})^{-1}  W_{x,x+1}
 \right\rangle\;,
\end{equation}
\begin{equation}
\label{V2.2}
\frac{3\hat{a}(y)^2C_H |B_{l+1}|}{N} \sum_{x \in B_{l+1}} \left\langle  p^2_{x+1}-p^2_x
\;,\;(-\mathcal L_{N})^{-1} p^2_{x+1}-p^2_x
 \right\rangle\;,
\end{equation}
\begin{equation}
\label{V2.3}
\frac{3C_H |B_{l+1}|}{N} \sum_{x \in B_{l+1}} \left\langle (-\mathcal{L}_N) \tau^xF_k
, \tau^xF_k
\right\rangle\;.
\end{equation}
Here $C_H$ represents a constant depending on $H$ and $T$, that can be multiplied by a constant from line to line. 

Using the variational formula for the $\mathscr H_{-1}$ norm (see (\ref{variational-1})) we can see that  the expression in (\ref{V2.1}) is equal to
$$
\frac{C_H |B_{l+1}| }{N} \sum_{x \in B_{l+1}}\sup_{g \in L^2(\nu^N_y)}\left\{ 2 \left\langle W_{x,x+1}, g \right\rangle + \left\langle g, \mathcal L_{x,x+1}g \right\rangle \right\}\;,
$$ 
where,
$$
\mathcal L_{x,x+1}=\frac{1}{2}X_{x,x+1}[a(p_x,p_{x+1})X_{x,x+1}]\;.
$$
From the definition given in (\ref{current}) we have
$W_{x,x+1}=-X_{x,x+1}[a(p_x,p_{x+1})p_xp_{x+1}]$. Performing integration by parts in the two inner products, we can write the quantity inside the sum as
$$
\frac{1}{2}\sup_{g \in L^2(\nu^N_y)}\left\{ 4 \left\langle a(p_x,p_{x+1})p_xp_{x+1} , X_{x,x+1}g \right\rangle - \left\langle X_{x,x+1}g,a(p_x,p_{x+1}) X_{x,x+1}g \right\rangle \right\}\;,
$$
which by the elementary inequality $2ab \leq A^{-1}a^2+ Ab^2$, is bounded above by
$$
2\left\langle a(p_x,p_{x+1})p^2_xp^2_{x+1}\right\rangle\;.
$$
Then the expression in (\ref{V2.1}) is bounded above by
$$
\frac{C_H |B_{l+1}|^2}{N}\;.
$$
The same is true for the term corresponding to (\ref{V2.2}), which coincides with (\ref{V2.1}) if we take $a(r,s)\equiv 1$.

Since $F_k$ is a local function supported in a box of size $s_k$ and $\nu_y^{N}$ is translation invariant, we have for all $x,y \in \mathbb T_N$
$$
\left\langle  \tau^xF_k  , (-\mathcal{L}_N) \tau^yF_k
\right\rangle
\leq
s_k\;\mathbb ||X_{0,1}(\widetilde {F_k})||_{L^2(\nu_y^{N})}^2\;,
$$ 
which implies that the expression in (\ref{V2.3}) is bounded by
$$
\frac{C_H |B_{l+1}|^2 }{N}s_k\;\mathbb ||X_{0,1}(\widetilde {F_k})||_{L^2(\nu_y^{N})}^2\;,
$$
ending the proof.
\end{proof}
\begin{lemma}
\begin{equation*}
\lim_{k \to \infty}\lim_{N \to \infty} \mathbb E_{\nu^N_y}\left[\sup_{0 \leq t \leq T}
\left(\int^t_0 V_3 \; ds\right)^2
\right]\;=\;0\;.
\end{equation*}
\end{lemma}
\begin{proof}
The proof is similar to the preceding one.
\end{proof}

\begin{lemma}
\label{lemmaboltzmangibbs}
\begin{equation*}
\lim_{k \to \infty}\lim_{N \to \infty} \mathbb E_{\nu^N_y}\left[\sup_{0 \leq t \leq T}
\left(\int^t_0 V_1 \; ds\right)^2
\right]\;=\;0\;.
\end{equation*}
\end{lemma}
\begin{proof}
Recall that the expectation in the last expression is by definition
{\small
$$
N\mathbb E_{\nu^N_y}\left[\sup_{0 \leq t \leq T}
\left(
\int^t_0\sum_{j=1}^{l}\left\{
\sum_{x \in B_j^{\prime}}\nabla_N H_s(x/N) V_x
-\sum_{x \in B_j^{k} }  \nabla_N H_s(x/N) \mathcal{L}_N \tau^xF_k\;
 ds \right\} \right)^2
\right].
$$
}
The smoothness of the function $H$ allows to replace $\nabla_N H_s(x/N)$ into each sum in the last expression by $\nabla_N H_s(x^{*}_j/N)$, where $x^{*}_j \in B_j$ (for instance, take $x^{*}_j=(j-1)K + 1$ ), obtaining
$$
N\mathbb E_{\nu^N_y}\left[\sup_{0 \leq t \leq T}
\left(
\int^t_0\sum_{j=1}^{l}\nabla_N H_s(x^{*}_j/N)\left\{
\sum_{x \in B_j^{\prime}} V_x
-\sum_{x \in B_j^{k} } \mathcal{L}_N \tau^xF_k\;
 ds \right\} \right)^2
\right].
$$
By proposition \ref{cenlimvarian} and Remark \ref{cenlimvarianremark}, the quantity in the preceding line is bounded above by 
$$
\frac{C_H}{N} \sum_{j=1}^{l}
 \left\langle   \sum_{x \in B_j^{\prime}} V_x - \sum_{x \in B_j^{k} } \mathcal{L}_N \tau^xF_k
\;,\;(-\mathcal L_{N})^{-1}  \sum_{x \in B_j^{\prime}} V_x - \sum_{x \in B_j^{k} } \mathcal{L}_N \tau^xF_k
 \right\rangle_{\nu_y}\;,
$$
Using the variational formula for the $\mathscr H_{-1}$ norm given in (\ref{variational-1}) and the convexity
of the Dirichlet form, we are able to replace $(-\mathcal L_{N})^{-1}$ by $(-\mathcal L_{B_j^{\prime}})^{-1}$ in the expression above. In addition, by translation invariance of the measure $\nu^N_y$ we can bound this expression by 
$$
C_H\frac{l}{N}
\left\langle   \sum_{x \in B_1^{\prime}} V_x - \sum_{x \in B_1^{k} } \mathcal{L}_N \tau^xF_k
\;,\;(-\mathcal L_{B_1^{\prime}})^{-1}  \sum_{x \in B_1^{\prime}} V_x - \sum_{x \in B_1^{k} } \mathcal{L}_N \tau^xF_k
 \right\rangle_{\nu_y}.
$$
By the equivalence of ensembles stated in Appendix \ref{equiensem} and the fact that $\frac{l}{N} \sim \frac{1}{M}$, the  limit superior, as $N$ goes to infinity, of the last expression is bounded above by
{\small
$$
C_H \limsup_{M\to \infty}\frac{1}{M}
\left\langle   \sum_{x \in B_1^{\prime}} V_x - \sum_{x \in B_1^{k} } \mathcal{L}_N \tau^xF_k
\;,\;(-\mathcal L_{B_1^{\prime}})^{-1}  \sum_{x \in B_1^{\prime}} V_x - \sum_{x \in B_1^{k} } \mathcal{L}_N \tau^xF_k
 \right\rangle_{\nu_{M,\sqrt{M}y}}\!\!\!\!\!\!\!\!\!\!\!\!\!\!\!\!.
$$
}
The last line can be written as
\begin{equation}
\label{vartotal}
C_H \limsup_{M\to \infty}\frac{1}{M}\Delta_{M,y}(B_M+\widehat{a}(y)A_M-H_M^{F_k},\; B_M+\widehat{a}(y)A_M-H_M^{F_k}) ,
\end{equation}
by using the notation introduced in Section \ref{CLTVandDC}. For that, it suffices to replace M by 2M+1 from the beginning of this section. Here $B_M$ correspond to the current in a block, and is not to be confused with the notation used for the blocks themselves.

On the other hand, it is easy to check that the variance appearing in (\ref{vartotal}) is equal to
{\small
\begin{align*}
(\widehat{a}(y))^2\Delta_{M,y}(A_M,A_M) &+ \Delta_{M,y}(B_M,B_M) +\Delta_{M,y}(H^{F_k}_M,H^{F_k}_M)\\
&-2\widehat{a}(y)\Delta_{M,y}(A_M,B_M) + \Delta_{M,y}(A_M,H^{F_k}_M) -\Delta_{M,y}(B_M,H^{F_k}_M)\;.
\end{align*}
}
Therefore, thanks to Theorem \ref{thecovariances}  and Theorem \ref{covarA_N}, if we divide by $M$ and take the limit as $M$ goes to infinity at both sides of last expression, we can conclude that
\begin{align*}
\lim_{ M \to \infty}\frac{1}{M}\Delta_{M,y}(B_M+\widehat{a}(y)A_M-H_M^{F_k})&=
4a(y,1/2{F_k})-4y^4\widehat{a}(y)\;.
\end{align*}
By the definition of the sequence $\{F_k\}_{k \geq  1}$  (see (\ref{sequencetohat})), the limit as $k$ goes to infinity of the last term is equal to zero.
\end{proof}

\section{Tightness}
\label{Tightness}
Let us firstly introduce some notation in order to define a space in which fluctuations take place and in which we will be able to prove tightness. Let $-\Delta$ be the positive operator, essentially self-adjoint on $L^2([0,1])$ defined by
\begin{align*}
\textit{Dom}(-\Delta)&=C_0^2([0,1]),\\
-\Delta&= -\frac{d^2}{dx^2}\;,
\end{align*}
where $C_0^2([0,1])$ denotes the space of twice continuously differentiable functions on $(0,1)$ which are continuous in $[0,1]$ and which vanish at the boundary. It is well known that its normalized eigenfunctions are given by  $e_n(x)=\sqrt 2 \sin(\pi nx)$ with corresponding eigenvalues $\lambda_n = (\pi n)^2$ for every $n\in \mathbb N\ $, moreover, $\{e_n\}_{n\in \mathbb N}$ forms an orthonormal basis of $L^2([0,1])$. 

For any nonnegative $k$ denote by $\mathcal H_k$ the Hilbert space obtained as the completion of $C_0^2([0,1])$ endowed with the inner product
$$
\langle f,g \rangle_k =\langle f,(-\Delta)^kg \rangle\;,  
$$ 
where $\langle , \rangle$ stands for the inner product in $L^2([0,1])$. We have from the spectral theorem for self-adjoint operators that 
\begin{equation}
\label{defH_k}
\mathcal H_k = \{f \in L^2([0,1]): \sum_{n=1}^{\infty}n^{2k}\langle f,e_n\rangle < \infty \}\;,
\end{equation}
and
\begin{equation}
\label{innerH_k}
\langle f,g \rangle_k = \sum_{n=1}^{\infty}(\pi n)^{2k}\langle f,e_n\rangle\langle g,e_n\rangle \;.
\end{equation}
This is valid also for negative $k$. In fact, if we denote the topological dual of $\mathcal H_k$
 by $\mathcal H_{-k}$ we  have
\begin{equation}
\label{defH_{-k}}
\mathcal H_{-k} = \{f \in \mathcal D'([0,1]): \sum_{n=1}^{\infty}n^{-2k}f(e_n)^2 < \infty \}\;.
\end{equation}
The $\mathcal H_{-k}$-inner product between the distributions $f$ and $g$ can be written as
\begin{equation}
\label{innerH_{-k}}
\langle f,g \rangle_{-k} = \sum_{n=1}^{\infty}(\pi n)^{-2k} f(e_n) g(e_n) \;,
\end{equation}

Denote by $\mathbb Q_N$ the probability measure on $C([0,T],\mathcal H_{-k})$ induced by the energy fluctuation field $Y_t^N$ and the Markov process $\{p^N(t),t \geq 0\}$ defined in Section \ref{NotandRes}, starting from the equilibrium probability measure $\nu^N_{y}$.

We are now ready to state the main result of this section, which proof is divided in lemmas. 
\begin{theorem}
\label{Compactness}
The sequence $\{\mathbb{Q}_{N}\}_{N\geq 1}$ is tight in $C([0,T],\mathcal H_{-k})$ for $k>\frac{3}{2}$ .
\end{theorem}

In order to establish the tightness of the sequence $\{\mathbb{Q}_{N}\}_{N\geq 1}$ of probability measures on  $C([0,T],\mathcal H_{-k})$, it suffices to check the following two conditions (\textit{c.f.} \cite{KL} p.299),
\begin{align}
 &
\label{tight1}
\lim_{A \to \infty}\limsup_{N \to \infty}\mathbb P_{\nu_y}\left[\sup_{0_\leq t \leq T} ||Y^N_t||_{-k}>A\right] = 0\;,\\
 & 
\label{tight2}
 \lim_{\delta \to 0}\limsup_{N \to \infty}\mathbb P_{\nu_y}\left[w(Y^N,\delta)> \epsilon\right] = 0\;,
\end{align}
where the modulus of continuity $w(Y,\delta)$ is defined by
\[
w(Y,\delta)=\sup_{\substack{|t-s|<\delta\\0\leq s <t\leq T}}||Y_t - Y_s||_{-k}\;.
\]
Let us recall that for every function  $H\in C^2(\mathbb{T})$ we have
\begin{equation}
\label{twoterms}
Y_t^N(H)\;=\; Y_0^N(H) - Z_t^N(H)- M_t^N(H)\;,
\end{equation}
where,
\begin{align*}
Z_t^N(H) &=\; \int^t_0 \sqrt{N} \sum_{x \in \mathbb T_N}\nabla_N H(x/N) W_{x,x+1}(s)ds\;,\\
M_t^N(H)& =\; \int^t_0 \frac{1}{\sqrt{N}} \sum_{x \in \mathbb T_N}\nabla_N H(x/N)\sigma(p_x(s),p_{x+1}(s))\;dB_{x,x+1}(s)\;.
\end{align*} 
The quadratic variation of the martingale $\{M_t^N(H)\}_{t\geq 0}$ is given by
\begin{equation*}
\langle M_t^N(H)\rangle(t)=\frac{1}{N}\sum_{x\in\mathbb{T}_N}\int^{t}_{0}|\nabla_N H(\frac{x}{N},s)|^2a(p_x,p_{x+1})p^2_xp^2_{x+1}ds\;.
\end{equation*}
We begin by giving the following key estimate. 
\begin{lemma} There exist a constant $B=B(y,T)$ such that for every function  $H\in C^2(\mathbb{T})$ and every $N \geq 1$  
\label{keyestimate}
\begin{equation*}
\mathbb E_{\nu_y}\left[\sup_{0\leq t\leq T} (Y_t^N(H))^2\right]\leq
 B \left\{\frac{1}{N}\sum_{x \in \mathbb T_N}H(x/N)^2 + \frac{1}{N}\sum_{x \in \mathbb T_N}(\nabla_N H(x/N))^2 \right\}\;.
\end{equation*}
\end{lemma}
\begin{proof} From the definition of the fluctuation field it is clear that
\begin{equation}
\mathbb E_{\nu_y}\left[(Y_0^N(H))^2\right] = 2y^4\frac{1}{N}\sum_{x \in \mathbb T_N}H(x/N)^2\;,
\end{equation}
and by Doob's inequality together with the fact that $a(\cdot,\cdot)\leq C$ we have
\begin{equation}
\mathbb E_{\nu_y}\left[\sup_{0\leq t\leq T} (M_t^N(H))^2\right] \leq CTy^4\frac{1}{N}\sum_{x \in \mathbb T_N}(\nabla_N H(x/N))^2\;.
\end{equation}
From Proposition \ref{cenlimvarian} of Section \ref{Boltzmann-Gibbs} and the variational formula given in (\ref{variational-1}) we obtain
{\small
\begin{equation*}
\mathbb E_{\nu_y}\left[
\sup_{0\leq t\leq T}\left(Z_t^N(H)\right)^2\right]
\leq
\frac{24T}{N}\sup_{g \in D(\mathcal L)}\left\{\langle 2\sum_{x \in \mathbb T_N}\nabla_N H(x/N) W_{x,x+1}g\rangle_{\nu_y} +
\langle g,\mathcal Lg\rangle_{\nu_y}
\right\}\;.
\end{equation*}
}
After integration by parts, the first term in the expression into braces can be written as
\begin{equation*}
-2\langle \sum_{x \in \mathbb T_N}\nabla_N H(x/N) a(p_x,p_{x+1})p_xp_{x+1}X_{x,x+1}(g)\rangle_{\nu_y}\;,
\end{equation*}
which by Schwartz inequality is bounded above by
\begin{equation*}
2\langle \sum_{x \in \mathbb T_N}(\nabla_N H(x/N))^2 a(p_x,p_{x+1})p^2_xp^2_{x+1}\rangle_{\nu_y}
+
\frac{1}{2}\langle \sum_{x \in \mathbb T_N} a(p_x,p_{x+1})(X_{x,x+1}(g))^2\rangle_{\nu_y}.
\end{equation*}
Thus,
\begin{equation*}
\mathbb E_{\nu_y}\left[
\sup_{0\leq t\leq T}\left(Z_t^N(H)\right)^2\right]
\leq
48TCy^4\frac{1}{N}\sum_{x \in \mathbb T_N}(\nabla_N H(x/N))^2\;.
\end{equation*}
\end{proof}
\begin{corollary}
Condition (\ref{tight1}) is valid for $k>\frac{3}{2}$.
\end{corollary}
\begin{proof}
From (\ref{innerH_{-k}}) and Lemma \ref{keyestimate} we obtain
\begin{align*}
\limsup_{N \to \infty}\mathbb E_{\nu_y}\left[\sup_{0_\leq t \leq T} ||Y^N_t||^2_{-k}\right]&\leq 
 \sum_{n=1}^{\infty}(\pi n)^{-2k}\limsup_{N \to \infty} \mathbb E_{\nu_y}\left[\sup_{0\leq t\leq T} Y_t^N(e_n)^2\right]\\
&\leq B \sum_{n=1}^{\infty}(\pi n)^{-2k} (1+(\pi n)^2)\;.
\end{align*}
The proof is then concluded by using Chebychev's inequality.
\end{proof}
In view of (\ref{innerH_{-k}}) and Lemma \ref{keyestimate} we reduce the problem of equicontinuity as follows.
\begin{align*}
\limsup_{N \to \infty}\mathbb E_{\nu_y}\left[w(Y^N,\delta)\right]&\leq 
 \sum_{n=1}^{\infty}(\pi n)^{-2k} \limsup_{N \to \infty}\mathbb E_{\nu_y}\left[\sup_{\substack{|t-s|<\delta\\0\leq s <t\leq T}}(Y^N_t(e_n) - Y^N_s(e_n))^2\right]\\
&\leq 4 \sum_{n=1}^{\infty}(\pi n)^{-2k} \limsup_{N \to \infty}\mathbb E_{\nu_y}\left[\sup_{\substack{|t-s|<\delta\\0\leq s <t\leq T}}(Y^N_t(e_n))^2\right]\\
&\leq B \sum_{n=1}^{\infty}(\pi n)^{-2k} (1+(\pi n)^2)\;.
\end{align*}
Therefore, the series appearing in the first line of the above expression is uniformly convergent in $\delta$ if $k>\frac{3}{2}$. Thus, in order to verify condition (\ref{tight2}) it is enough to prove 
\begin{equation*}
 \lim_{\delta \to 0}\limsup_{N \to \infty}\mathbb E_{\nu_y}\left[\sup_{\substack{|t-s|<\delta\\0\leq s <t\leq T}}(Y^N_t(e_n) - Y^N_s(e_n))^2\right]
 = 0\;,
\end{equation*}
for every $n\geq 1$.

We analyze separately the terms corresponding to  $M_t^N$ and $Z_t^N$ (see \ref{twoterms}).  In next lemma we state a global estimate for the martingale part. 
\begin{lemma}
\label{martfluctvaria}
For every function $H$ and every $m \in\mathbb N$, there exists a constant $C$ depending only on $m$ such that
\begin{equation*}
\mathbb E_{\nu^N_y} [\;| M_t^N(H)|^{2m}\;] \leq C y^{2m} t^m
 \left\{ \frac{1}{N}\sum_{x\in\mathbb{T}_N}|\nabla_N H(\frac{x}{N},s)|^2 \right\}^m\:.
\end{equation*}
\end{lemma}
\begin{proof}
Denote the continuous martingale $M_t^N(H)$ by $M_t$, and let $C_m$ be a constant depending only on $m$ which can change from line to line.

Using the explicit expression for the martingale and applying It\^o's formula we have
$$
d(M_t)^{2m}= 2m(M_t)^{2m-1}dM_t + m(2m-1)(M_t)^{2m-2}Q_tdt,
$$
where, 
$$
Q_t=\frac{1}{N}\sum_{x\in\mathbb{T}_N}|\nabla_N H(\frac{x}{N})|^2a(p_x,p_{x+1})p^2_xp^2_{x+1}\;.
$$
Explicit calculations lead us to
$$
\mathbb E_{\nu^N_y}[(Q_t)^m]\leq C_m y^{2m}\big\{\frac{1}{N}\sum_{x\in\mathbb{T}_N}|\nabla_N H(\frac{x}{N})|^2\big\}^m\;,
$$
thus, by stationarity and applying H\"older's inequality for space and time we obtain
\begin{equation}
\label{holder}
\mathbb E_{\nu^N_y}[(M_t)^{2m}] \leq C_my^2t^{\frac{1}{m}}
\frac{1}{N}\sum_{x\in\mathbb{T}_N}|\nabla_N H(\frac{x}{N})|^2
\left(\int_0^t\mathbb E_{\nu^N_y}[(M_s)^{2m}]ds\right)^{\frac{2m-2}{2m}}
\;.
\end{equation}
In terms of the function 
$
f(t)=\left(\int_0^t\mathbb E_{\nu^N_y}[(M_s)^{2m}]ds\right)^{\frac{1}{m}}
$,
inequality (\ref{holder}) reads 
$$
f'(t) \leq C_my^2t^{\frac{1}{m}}
\frac{1}{N}\sum_{x\in\mathbb{T}_N}|\nabla_N H(\frac{x}{N})|^2 \;,
$$
and integrating we obtain
$$
f(t) \leq C_my^2t^{1+\frac{1}{m}}
\frac{1}{N}\sum_{x\in\mathbb{T}_N}|\nabla_N H(\frac{x}{N})|^2 \;.
$$
The proof ends by using the last line to estimate the right hand side of (\ref{holder}).
\end{proof} 
In order to pass from this global estimate to a local estimate, we will use the Garcia's inequality.

\begin{lemma}
\textit{(Garcia-Rodemich-Rumsey)} (\textit{cf} \cite{SV} p.47)
\label{GRR}
Given $p:[0, \infty]\to \mathbb{R}$  continuous, strictly increasing functions such that $p(0)=\psi(0)=0$ and $\lim_{t\to \infty}\psi (t)=\infty.$  Given $\phi\in C([0,T]; \mathbb{R}^d)$ if
\[
\int_0^T\int_0^T\psi \left(\frac{|\phi(t)-\phi(s)|}{p(|t-s|)}\right)dsdt \leq B < \infty
\]
then, for $0\leq s<t\leq T$
\[
|\phi(t)-\phi(s)|\leq 8 \int_0^{t-s} \psi ^{-1}\left(\frac{4B}{u^2}\right)p(du).
\]
\end{lemma}
\begin{lemma}
\label{martfluctvaria2}
For every function $H \in C^2(\mathbb T)$, 
\begin{equation*}
 \lim_{\delta \to 0}\limsup_{N \to \infty}\mathbb E_{\nu_y}\left[
\sup_{\substack{|t-s|<\delta\\0\leq s <t\leq T}}
\left( M_t^N(H)-M_s^N(H)\right)^2
\right] = 0\;.
\end{equation*}
\end{lemma}
\begin{proof}
Taking $p(u)=\sqrt{u}$ and $\psi(u)= u^6$ in Lemma \ref{GRR} we get  
\begin{equation*}
|\phi(t)-\phi(s)| \leq  C B^{1/6}|t-s|^{1/6}\;,
\end{equation*}
where,
\begin{equation}
\label{B}
B = \int_0^T\int_0^T \frac{|\phi(t)-\phi(s)|^6}{|t-s|^3}dsdt\;.
\end{equation}
Taking $\phi(t)= M_t^N(H)$ we obtain
\begin{equation*}
\mathbb E_{\nu^N_y}\big[
  \sup_{\substack{|t-s|<\delta\\0\leq s <t\leq T}}|\widehat M^H_N(t)-\widehat M^H_N(s)|^2
\big] \leq
Cy^2\delta^{1/3} T^{2/3}\frac{1}{N}\sum_{x \in \mathbb T_N} (\nabla H(x/N))^2\;,
\end{equation*}
which implies the desired result.

Observe that the integral in (\ref{B}) is finite, which permits to apply Lemma \ref{GRR}. In fact, as a consequence of Lemma \ref{martfluctvaria} and Kolmogorov- \v{C}entsov theorem, we have $\alpha$-H\"older continuity of paths for $\alpha \in [0,\frac{1}{2})$.
\end{proof} 
The proof of Theorem \ref{Compactness} will be concluded by proving the following lemma. 
\begin{lemma}For every function $H \in C^2(\mathbb T)$, 
\begin{equation*}
 \lim_{\delta \to 0}\limsup_{N \to \infty}\mathbb E_{\nu_y}\left[
\sup_{\substack{|t-s|<\delta\\0\leq s <t\leq T}}
\left(Z_t^N(H)-Z_s^N(H)\right)^2
\right] = 0
\end{equation*}
\end{lemma}
\begin{proof}
Recall that the expectation appearing above is by definition
\begin{equation}
\label{Z_t^Nterm}
\mathbb E_{\nu_y}\left[
\sup_{\substack{|t-s|<\delta\\0\leq s <t\leq T}}
\left(\int^t_s \sqrt{N} \sum_{x \in \mathbb T_N}\nabla_N H(x/N) W_{x,x+1}(s)ds \right)^2
\right]\;.
\end{equation}
Now we take advantage of the decomposition obtained for the current in the preceding sections, which allows to study separately the diffusive part of the current and the part coming from a fluctuation term. For this we add and subtract $\hat{a}(y)[p^2_{x+1}-p^2_{x}]+\mathcal L_N(\tau^xF_k(p))\;$ from $\;W_{x,x+1}$, obtaining that (\ref{Z_t^Nterm}) is bounded above by 3 times the following sum 
{\small
\begin{equation*}
4\mathbb E_{\nu_y}\!\!\left[\!
\sup_{0\leq t \leq T} \!\!
\left(\int^t_0 \sqrt{N} \sum_{x \in \mathbb T_N}\nabla_N H(x/N)
[ W_{x,x+1}(s)-\hat{a}(y)[p^2_{x+1}-p^2_{x}]-\mathcal L_N(\tau^xF_k(p))]ds \! \right)^2
\right]
\end{equation*}
\begin{equation*}
+ \; \hat{a}(y)^2\mathbb E_{\nu_y}\left[
\sup_{\substack{|t-s|<\delta\\0\leq s <t\leq T}}
\left(\int^t_s \sqrt{N} \sum_{x \in \mathbb T_N}\nabla_N H(x/N)[p^2_{x+1}-p^2_{x}]ds \right)^2
\right]
\end{equation*}
\begin{equation*}
+\; \mathbb E_{\nu_y}\left[
\sup_{\substack{|t-s|<\delta\\0\leq s <t\leq T}}
\left(\int^t_s \sqrt{N} \sum_{x \in \mathbb T_N}\nabla_N H(x/N) \mathcal L_N(\tau^xF_k(p))ds \right)^2
\right]\;.
\end{equation*}
}
The first term tends to zero as k tends to infinity after N. In fact, this is the content of the Boltzmann-Gibbs Principle proved in Section \ref{Boltzmann-Gibbs}.

Performing a sum by parts and using Schwartz inequality together with the stationarity, we can see that the second term is bounded above by
\begin{equation*}
\hat{a}(y)^2\delta T E_{\nu_y}\left[
\left(\frac{1}{\sqrt{N}} \sum_{x \in \mathbb T_N}\Delta_N H(x/N)p^2_{x}\right)^2 
\right]\;.
\end{equation*}
We can replace in the last line $p_x^2$ by $[p_x^2-y^2]$ (because of periodicity), obtaining that this expression is bounded above by 
\begin{equation*}
3y^4\hat{a}(y)^2\delta T 
\frac{1}{{N}} \sum_{x \in \mathbb T_N}(\Delta_N H(x/N))^2
\;.
\end{equation*}
For the third term we add and subtract $M^{2}_{N,F_k}(H^{\cdot t})-M^{2}_{N,F_k}(H^{\cdot s})$ to the integral, where $M^{2}_{N,F_k}(H^{\cdot t})$ is the martingale defined after equation (\ref{redependingtimemart}). In that way we obtain that this term is bounded above by 2 times the sum of the following two terms,
\begin{equation*}
\mathbb E_{\nu_y}\left[
\sup_{\substack{|t-s|<\delta\\0\leq s <t\leq T}}
\left(M^{2}_{N,F}(H^{\cdot t})-M^{2}_{N,F}(H^{\cdot s}) \right)^2
\right]\;,
\end{equation*}
\begin{equation*}
4\mathbb E_{\nu_y}\left[
\sup_{0\leq t\leq T}
\left(M^{2}_{N,F_k}(H^{\cdot t})+\int^t_0 \sqrt{N} \sum_{x \in \mathbb T_N}\nabla_N H(x/N) \mathcal L_N(\tau^xF_k(p))ds \right)^2
\right]\;.
\end{equation*}
Since the functions $F_k$ are local and belong to the Schwartz space, we can handle the first term in the same way as we did with $M_t^N(H)$ in Lemma \ref{martfluctvaria} and Lemma \ref{martfluctvaria2} . The second term tends to zero as $N$ goes to infinity, as stated in Lemma \ref{twoterms1}.
\end{proof}

\section{The Space $\mathcal{H}_y$}
\label{{H}_y}
The aim of this section is to define the space $\mathcal{H}_y$ and prove the characterization that was used in the proof of Theorem \ref{covarA_N}. Let us begin by introducing some notation.

Let $\Omega=\mathbb{R}^\mathbb{Z}$ and $p=(\cdots,p_{-1},p_0,p_1,\cdots)$ a typical element of this set. Define for $i\in \mathbb{Z}$ the shift operator $\tau^{i}:\Omega\to\Omega$  by $\tau^{i}(p)_j=p_{j+i}$ , and $\tau^{i}f(p)=f(\tau^{i}p)$ for any function  $f:\Omega\to \mathbb{R}$. We will consider the product measure $\nu_y$ on $\Omega$ given by $d\nu_y=\prod_{-\infty}^{\infty}\frac{\exp(\frac{-p_x^2}{2y^2})}{\sqrt{2\pi} y}dp$.

Let us define $\mathcal{A}=\cup_{k\geq 1} \mathcal{A}_k$, where $\mathcal{A}_k$ is the space of smooth functions $F$ depending on $2k+1$ variables. Given $F\in\mathcal{A}_k$ we can consider the formal sum
\begin{equation}
\label{formalsum}
\widetilde{F}(p)=\sum_{j=-\infty}^{\infty}\tau^jF(p)\;,
\end{equation}
and for $i\in\mathbb{Z}$ the well defined
\[
\frac{\partial\widetilde{F}}{\partial p_i}(p)=\sum_{i-k\leq j\leq i+k}\frac{\partial}{\partial p_i}F(p_{j-k},\cdots,p_{j+k})\;.
\]
The formal invariance $\widetilde{F}(\tau (p))=\widetilde{F}(p)$ lead us to the precise covariance
\begin{equation}\label{covariance}
\frac{\partial\widetilde{F}}{\partial p_i}(p)=\frac{\partial\widetilde{F}}{\partial p_0}(\tau^i p)\;.
\end{equation}
Recall that $X_{i,j}=p_j\partial_{p_i}-p_i\partial_{p_j}$. Given $F\in\mathcal{A}$ and $i\in\mathbb{Z}$, $X_{i,i+1}(\widetilde{F})$ is well defined and satisfies
\[
X_{i,i+1}(\widetilde{F})(p)=\tau^iX_{0,1}(\widetilde{F})(p)\;.
\]
Finally we define the following set
\[
\mathcal{B}_y=\{X_{0,1}(\widetilde{F})\in L^2(\nu_y): F\in\mathcal{A}\}.
\]
In terms of the notation introduced above, the variational formula obtained in (\ref{^a1}) for the diffusion coefficient can be written as 
\begin{equation}
\label{coeffdifusion}
\hat{a}(y)=y^{-4}\inf_{\xi\in\mathcal{B}_y} \mathbb{E}_{\nu_y}[a(p_0,p_1)(p_0p_1+\xi)^2]\;.
\end{equation}
As is well known, if we denote by $\mathcal{H}_y$ the closure of $\mathcal{B}_y$ in $L^2(\nu_y)$, then
\[
\hat{a}(y)=y^{-4}\inf_{\xi\in\mathcal{H}_y} \mathbb{E}_{\nu_y}[a(p_0,p_1)(p_0p_1+\xi)^2]\;,
\]
and the infimum will be attained at a unique $\xi\in\mathcal{H}_y$.

At the end of the proof of Theorem \ref{covarA_N} we used an intrinsic characterization of the space $\mathcal{H}_y$. In order to obtain such a characterization, we can first observe that defining  $\xi=X_{0,1}(\widetilde{F})$ for $F\in \mathcal{A}$, the following properties are satisfied:
\begin{enumerate}
\item[\textit{i})]$\mathbb{E}_{\nu_y}[\xi]=0$,
\item[\textit{ii})]$\mathbb{E}_{\nu_y}[p_0p_1\xi]=0$,
\item[\textit{iii})] $X_{i,i+1}(\tau^j\xi)=X_{j,j+1}(\tau^i\xi)$ \ \ \ \ if\ \ \ \  $\{i,i+1\}\cap\{j,j+1\}=\emptyset$, 
\item [\textit{iv})] $p_{i+1}[X_{i+1,i+2}(\tau^i\xi)-X_{i,i+1}(\tau^{i+1}\xi)]=p_{i+2}\tau^i\xi-p_{i}\tau^{i+1}\xi$\ \ for $i\in \mathbb{Z}$.
\end{enumerate}
Now we can claim the desired characterization.
\begin{theorem}
\label{characterizationofH_y}
If $\xi\in L^2(\nu_y)$ satisfies conditions i) to iv) \textnormal{(the last two in a weak sense)} then $\xi\in\mathcal{H}_y$.
\end{theorem}
The proof of Theorem \ref{characterizationofH_y} relies on the results obtained in Appendix \ref{geomconsider} and Appendix \ref{spectralgap}.  Additionally the introduction of a cut off function is required in order to control large energies.
\begin{proof}
The goal is to find a sequence $(F_N)_{N\geq 1}$ in $\mathcal{A}$, such that the sequence  $\{X_{0,1}(\widetilde{F_N})\}_{N \geq 1}$ converges to $\xi$ in $ L^2(\nu_y)$. As is well known, the strong and the weak closure of a subspace of a Banach space coincide, therefore it will be enough to show that $\{X_{0,1}(\widetilde{F_N})\}_{N \geq 1}$ converges weakly to $\xi$ in $ L^2(\nu_y)$.

Firstly observe that for any smooth function $F(p_{-k},\cdots, p_{k})$ we can rewrite $X_{0,1}(\widetilde{F})$, by using (\ref{covariance}), as
\begin{equation}\label{suminvgrad}
\sum_{i=-k}^{k-1}X_{i,i+1}(F)(\tau^{-i}p)+\left(p_{k+1}\frac{\partial F}{\partial p_{k}} \right)(\tau^{-k}p)-\left(p_{-k-1}\frac{\partial F}{\partial p_{-k}} \right)(\tau^{k+1}p)\;.
\end{equation}
Roughly speaking, the idea is to use the criteria obtained in Appendix \ref{geomconsider} to integrate the system (\ref{poissonequation})  in order to find a function $F$ such that $\xi$ is approximated by the sum in the first term of (\ref{suminvgrad}), and then to control the border terms.\\
We define
\begin{equation}
\label{introcutoff}
\xi_{i,i+1}^{(2N)}=\mathbb{E}_{\nu_y}[\xi_{i,i+1}|\mathfrak{F}^{2N}_{-2N}]\varphi\left(\frac{1}{4N+1}\sum_{i=-2N}^{2N} p_i^2
\right)\;,
\end{equation}
where $\xi_{i,i+1}(p)=\tau_i\xi(p)$, $\mathfrak{F}^N_{-N}$ is the sub $\sigma$-field of $\Omega$ generated by $\{p_{-N},\cdots,p_N\}$ and $\varphi$ is a smooth function with compact support such that $0 \leq \varphi \leq 1$ and $\varphi(y^2)=1$. We introduce this cutoff in order to do uniform bounds later.

Since $\nu_y$ is a product measure and the part corresponding to $\varphi$ is radial, the set of functions $\{\xi_{i,i+1}^{(2N)}\}_{-2N\leq i\leq i+1\leq 2N}$ even satisfies conditions iii) and iv). Therefore, after Theorem \ref{inttheorem}  the system

\begin{equation}\label{system}
X_{i,i+1}(g^{(N)})=\xi_{i,i+1}^{(N)}\ \  \textit{for} \ \ -2N\leq i\leq i+1\leq 2N
\end{equation}
can be integrated. Since  $\mathbb{E}_{\nu_y}[g^{(N)}|p^2_{-2N}+\cdots+p^2_N]$ is radial and the integration was performed over spheres, $\widetilde{g}^{(2N)}=g^{(2N)}-\mathbb{E}_{\nu_y}[g^{(2N)}|p^2_{-2N}+\cdots+p^2_N]$  is still a solution of the system (\ref{system}). Therefore, without lost of generality, we can suppose that $\mathbb{E}_{\nu_y}[g^{(N)}|p^2_{-2N}+\cdots+p^2_{2N}=y^2]=0$ for every $y\in \mathbb{R}^+$. This will be useful when applying the spectral gap estimate.

In order to construct the desired sequence firstly define
\[
{g}^{(N,k)}=\frac{1}{2(N+k)y^4}\mathbb{E}_{\nu_y}[p^2_{-N-k-1}p^2_{N+k+1}g^{(2N)}|\mathfrak{F}^{N+k}_{-N-k}]\;,
\]
and,
\[
\widehat{g}^{N}(p_{-7N/4},\cdots,p_{7N/4})=\frac{4}{N}\sum_{k=N/2}^{3N/4}{g}^{(N,k)}\;.
\]
Using (\ref{suminvgrad}) for ${g}^{(N,k)}$ and averaging over $k$ we obtain that
\[
X_{0,1}\left(\sum_{j=-\infty}^{\infty}\tau^j\widehat{g}^{N}\right)=
\xi+y^{-4}\{I^1_N + I^2_N + I^3_N + R^1_N - R^2_N\},
\]
where,
{\small
\begin{align*}
I^1_N&=\widehat{\sum_{k=N/2}^{3N/4}}\ \ \widehat{\sum_{i=-N-k}^{N+k}}\tau^{-i}\mathbb{E}_{\nu_y}[p^2_{N+k+1}p^2_{-N-k-1}(\xi^{(2N)}_{i,i+1}-\xi^{(N+k)}_{i,i+1})\varphi(r^2_{-2N,2N})|\mathfrak{F}^{N+k}_{-N-k}]\;, 
\\
I^2_N&=\widehat{\sum_{k=N/2}^{3N/4}}\ \ \widehat{\sum_{i=-N-k}^{N+k}}\tau^{-i}\{(\xi^{(N+k)}_{i,i+1}-\xi_{i,i+1})\mathbb{E}_{\nu_y}[p^2_{N+k+1}p^2_{-N-k-1}\varphi(r^2_{-2N,2N})|\mathfrak{F}^{N+k}_{-N-k}]\} \;,
\\
I^3_N&=\widehat{\sum_{k=N/2}^{3N/4}}\ \ \widehat{\sum_{i=-N-k}^{N+k}}\xi(p)\tau^{-i}\mathbb{E}_{\nu_y}[p^2_{N+k+1}p^2_{-N-k-1}(\varphi(r^2_{-2N,2N})-1)|\mathfrak{F}^{N+k}_{-N-k}] \;,
\\
R^1_N&=\widehat{\sum_{k=N/2}^{3N/4}}\tau^{-N-k}\{p_{N+k+1}\frac{\partial}{\partial p_{N+k}}g^{(N,k)}\}\;, 
\\
R^2_N&=\widehat{\sum_{k=N/2}^{3N/4}}\tau^{N+k+1}\{p_{-N-k-1}\frac{\partial}{\partial p_{-N-k}}g^{(N,k)}\}\;. 
\end{align*}
}
Here $\widehat{\sum_{k=N/2}^{3N/4}}=\frac{4}{N+4}\sum_{k=N/2}^{3N/4} $, that is, the hat over the sum symbol means that this sum is in fact an average. The notation $r^2_{-2N,2N}$ is just an abbreviation for $\frac{1}{4N+1}\sum_{i=-2N}^{2N} p_i^2$.

The proof of the theorem will be concluded in the following way. In Lemma \ref{middleterms}  the convergence in  $L^2(\nu_y)$ to zero of the middle terms $I^1_N,I^2_N,I^3_N$ is demonstrated.   We stress the fact that weak convergence to zero of each border term is false. However, weak convergence to zero of the sequence $\{R^1_N - R^2_N\}_{N \geq 1}$ is true, as ensured by Lemmas \ref{boundingborderterms}, \ref{subsequencelimits}  and \ref{weakconvergence}.

Therefore 
\[
\left\{ X_{0,1}\left(\sum_{j=-\infty}^{\infty}\tau^j\widehat{g}^{N}\right)\right\}_{N \geq 1}\;,
\]
is weakly convergent to $\xi$.
\end{proof}

Before entering in the proof of the lemmas, let us state two remarks. 
\begin{remark}
\label{condconv}
We know that $\mathbb{E}_{\nu_y}[\xi_{0,1}|\mathfrak{F}^N_{-N}]\xrightarrow{L^2}\xi_{0,1}$, \textit{i.e} given $\epsilon>0$ there exist $N_0\in \mathbb{N}$ such that
\[
\mathbb{E}_{\nu_y}[|\xi_{0,1}-\xi^{(N)}_{0,1}|^2]\leq \epsilon\ \  \textit{if}\ \ N\geq N_0\;.
\]
Moreover, by translation invariance we have
\[
\mathbb{E}_{\nu_y}[|\xi_{i,i+1}-\xi^{(N)}_{i,i+1}|^2]\leq \epsilon\ \  \textit{if}\ \ [-N_0-i,N_0+i]\subseteq [-N,N]\;.
\]
\end{remark}
\begin{proof}
Given $A\in \mathfrak{F}^{N-i}_{-N-i}$ we have
\begin{align*}
\int_{A}\xi^{(N)}_{i,i+1}(\tau^{-i}p)\nu_y(dp)
=\int_{A}\xi_{0,1}(p)\nu_y(dp)\;.
\end{align*}
Since in addition $\xi^{(N)}_{i,i+1}(\tau^{-i})\in \mathfrak{F}^{N-i}_{-N-i}$, we have 
\[ 
\xi^{(N)}_{i,i+1}(\tau^{-i})=\mathbb{E}_{\nu_y}[\xi_{0,1}|\mathfrak{F}^{N-i}_{-N-i}]\;,
\]
and therefore,
\[
\mathbb{E}_{\nu_y}[|\xi_{i,i+1}-\xi^{(N)}_{i,i+1}|^2]=\mathbb{E}_{\nu_y}[|\xi_{0,1}-\xi^{(N)}_{i,i+1}(\tau^{-i})|^2]\leq \mathbb{E}_{\nu_y}[|\xi_{0,1}-\xi^{(N_0)}_{0,1}|^2].
\]
\end{proof}
\begin{remark}
\label{stronglaw} 
A strong law of large numbers is satisfied for $(p^2_i)_{i\in \mathbb{Z}}$. In fact we have
\[
\mathbb{E}_{\nu_y}\left[\left(\frac{1}{N}\sum_{i=1}^{N} p_i^2-y^2\right)^2\right]\leq\frac{8y^8}{N}\;.
\]
\end{remark}
\begin{lemma}[\textbf{Middle terms}]
\label{middleterms}
For $i= 1,2,3$ we have
\begin{equation*}
\lim_{N \to \infty} \mathbb{E}_{\nu_y}[(I^i_N)^2] = 0 \;.
\end{equation*}
\end{lemma}
\begin{proof}
The convergence to zero as N tends to infinity of $I^1_N$ and $I^2_N$ in $L^2(\nu_y)$ follows directly from Schwartz inequality, Remark \ref{condconv} and the fact that $\varphi$ is a bounded function. \\
Using exchange symmetry of the measure, $I^3_N$ can be rewritten as
\[
\xi(p)\widehat{\sum_{k=N/2}^{3N/4}}\
\widehat{\sum_{i=-N-k}^{N+k}}\mathbb{E}_{\nu_y}[\widehat{\sum_{j=1}^{N-k}}p^2_{N+k+j}p^2_{-N-k-j}(\varphi(r^2_{-2N,2N})-1)|\mathfrak{F}^{N+k}_{-N-k}](\tau^{-i}p) \;,
\]
and then we decompose it as $J^1_N+y^2J^2_N$, where
{\small 
\[J^1_N(p)=
\xi(p)\widehat{\sum_{k=N/2}^{3N/4}}\
\widehat{\sum_{i=-N-k}^{N+k}}
\mathbb{E}_{\nu_y}[\widehat{\sum_{j=1}^{N-k}}\{p^2_{N+k+j}p^2_{-N-k-j}-y^4\}(\varphi(r^2_{-2N,2N})-1)|\mathfrak{F}^{N+k}_{-N-k}](\tau^{-i}p),
\]
}
and,
\[J^2_N(p)=
\xi(p)\widehat{\sum_{k=N/2}^{3N/4}}\
\widehat{\sum_{i=-N-k}^{N+k}}\mathbb{E}_{\nu_y}[\varphi(r^2_{-2N,2N})-1|\mathfrak{F}^{N+k}_{-N-k}](\tau^{-i}p). 
\]
Firstly observe that
\[
|J^1_N|^2\leq |\xi(p)|^2\widehat{\sum_{k=N/2}^{3N/4}}\
\widehat{\sum_{i=-N-k}^{N+k}}\mathbb{E}_{\nu_y}\left[\left(\widehat{\sum_{j=1}^{N-k}}\{p^2_{N+k+j}p^2_{-N-k-j}-y^4\}\right)^2\right].
\]
Being the expectation into last expression bounded by $\frac{8y^8}{N-k}$, we obtain
\[
||J^1_N||^2_{L^2(\nu_y)}\leq \frac{32y^4}{N}||\xi||^2_{L^2(\nu_y)}\;.
\]
On the other hand, writing explicitly the conditional expectation appearing in $J^2_N$ we see that
{\small
\[|J^2_N(p)|^2 \leq
|\xi(p)|^2\widehat{\sum_{k=N/2}^{3N/4}}\
\widehat{\sum_{i=-N-k}^{N+k}}\int |\varphi(\frac{1}{4N+1}\sum_{|j|>N+k}q^2_{j}+\frac{1}{4N+1}\sum_{|j|\leq N+k}p^2_{j+i})-1|^2\nu_y(dp).
\]
}
Rewrite the integral into last expression as
\[
\int |\varphi(\frac{1}{4N+1}\sum_{|j|>N+k}(q^2_{j}-y^2)+\frac{1}{4N+1}\sum_{|j|\leq N+k}(p^2_{j+i}-y^2)+y^2)-1|^2\nu_y(dp)\;.
\]
Using the fact that $\varphi$ is a Lipschitz positive function bounded by $1$ and satisfying $\varphi(y^2)=1$, we get that $|J^2_N|^2$ is bounded by
{\small
\[
|\xi(p)|^2\widehat{\sum_{k=N/2}^{3N/4}}\
\widehat{\sum_{i=-N-k}^{N+k}}1\wedge\int|\frac{1}{4N+1}\sum_{|j|>N+k} (q^2_{j}-y^2)+\frac{1}{4N+1}\sum_{|j|\leq N+k}(p^2_{j+i}-y^2)|^2\nu_y(dp)\;,
\]
}
where $a \wedge b$ denote the minimum of $\{a,b\}$. 

Therefore, taking expectation and using the strong law of large numbers together with the dominated convergence theorem, the convergence to zero as N tends to infinity of $I^3_N$ in $L^2(\nu_y)$ is proved.
\end{proof}

\begin{lemma}[\textbf{Bounding border terms}]
\label{boundingborderterms}
The sequences $\{R^i_N\}_{N \geq 1}$ are bounded in $L^2(\nu_y)$ for $i = 1,2$.
\end{lemma}
\begin{proof}
Recall that
\begin{equation*}
R^1_N=\widehat{\sum_{k=N/2}^{3N/4}}\frac{1}{2(N+k)}\tau^{-N-k}\{p_{N+k+1}\mathbb{E}_{\nu_y}[p^2_{-N-k-1}p^2_{N+k+1}\frac{\partial}{\partial p_{N+k}}g^{(2N)}|\mathfrak{F}^{N+k}_{-N-k}]\}\;.
\end{equation*}

Using the fact that $X_{N+k,N+k+1} = p_{N+k+1}\frac{\partial}{\partial p_{N+k}} - p_{N+k} \frac{\partial}{\partial p_{N+k+1}}$, we can rewrite last line as the sum of the following two terms. 
\begin{equation*}
\widehat{\sum_{k=N/2}^{3N/4}}\frac{1}{2(N+k)}\tau^{-N-k}\{p_{N+k+1}\mathbb{E}_{\nu_y}[p^2_{-N-k-1}p_{N+k+1}X_{N+k,N+k+1}g^{(2N)}|\mathfrak{F}^{N+k}_{-N-k}]\}
\end{equation*}
\begin{equation*}
\widehat{\sum_{k=N/2}^{3N/4}}\frac{1}{2(N+k)}\tau^{-N-k}\{p_{N+k}p_{N+k+1}\mathbb{E}_{\nu_y}[p^2_{-N-k-1}p_{N+k+1}\frac{\partial}{\partial p_{N+k+1}}g^{(2N)}|\mathfrak{F}^{N+k}_{-N-k}]\}\;.
\end{equation*}

By Schwartz inequality and (\ref{system}) we can see that the $L^2(\nu_y)$ norm of the first term is bounded by $\frac{y^{3}}{N}||\xi||_{L^2(\nu_y)}$. After integration by parts, the second term can be written as 

\begin{equation}
\label{secondterm}
\widehat{\sum_{k=N/2}^{3N/4}}\frac{1}{2(N+k)y^2}\tau^{-N-k}\{p_{N+k}p_{N+k+1}\mathbb{E}_{\nu_y}[p^2_{-N-k-1}(p^2_{N+k+1}-y^2)g^{(2N)}|\mathfrak{F}^{N+k}_{-N-k}]\}\;.
\end{equation}
Denote by $\sigma^{j,N+k+1}$ the interchange of coordinates $p_j$ and $p_{N+k+1}$. Using exchange invariance of the measure, we can see that the conditional expectation appearing in last expression is equal to
\[
\mathbb{E}_{\nu_y}[p^2_{-N-k-1}(p^2_{j}-y^2)(g^{(2N)}\circ\sigma^{j,N+k+1})|\mathfrak{F}^{N+k}_{-N-k}]\;,
\]
for $j = N+k+1, \cdots, 2N$. This permits to introduce a telescopic sum which will serve later to obtain an extra $\frac{1}{N}$ in order to use a spectral gap estimate.
Indeed, we decompose (\ref{secondterm}) as the sum of the following two terms.
\begin{equation}
\label{R21}
\widehat{\sum_{k=N/2}^{3N/4}}\frac{1}{2(N+k)y^2}\tau^{-N-k}
\mathbb{E}_{\nu_y}[p^2_{-N-k-1}\widehat{\sum_{j=N+k+1}^{2N}}(p^2_{j}-y^2)g^{(2N)}|\mathfrak{F}^{N+k}_{-N-k}]
\end{equation}
and
{\small
\begin{equation}
\label{R22}
\widehat{\sum_{k=N/2}^{3N/4}}\frac{1}{2(N+k)y^2}\tau^{-N-k}
\mathbb{E}_{\nu_y}[p^2_{-N-k-1}\widehat{\sum_{j=N+k+1}^{2N}}(p^2_{j}-y^2)(g^{(2N)}\circ\sigma^{j,N+k+1}-g^{(2N)})|\mathfrak{F}^{N+k}_{-N-k}]\;.
\end{equation}
}
By Schwartz inequality, the square of the conditional expectations appearing in last expressions are respectively bounded by
\[
CN^{-1}y^{12}\mathbb{E}_{\nu_y}[(g^{(2N)})^2|\mathfrak{F}^{N+k}_{-N-k}]
\]
and
\[
Cy^8\mathbb{E}_{\nu_y}[\widehat{\sum_{j=N+k+1}^{2N}}(g^{(2N)}\circ\sigma^{j,N+k+1}-g^{(2N)})^2|\mathfrak{F}^{N+k}_{-N-k}]\;,
\]
for a universal constant $C$.

Therefore, again by Schwartz inequality, we can see that the $L^2(\nu_y)$ norms of (\ref{R21}) and (\ref{R22}) are respectively bounded by 
\begin{equation}
\label{R31}
\frac{Cy^{10}}{N^3}\mathbb{E}_{\nu_y}\big[\big(\widehat{\sum_{k=N/2}^{3N/4}}p^2_{N+k}\big)(g^{(2N)})^2\big]
\end{equation}
and
\begin{equation}
\label{R32}
\frac{Cy^6}{N^2}\mathbb{E}_{\nu_y}\big[\big(\widehat{\sum_{k=N/2}^{3N/4}}p^2_{N+k}\big)
   \widehat{\sum_{j=3N/2+1}^{2N}}(g^{(2N)}\circ\sigma^{j,N+1}-g^{(2N)})^2\big]\;.
\end{equation}
Observe that $\widehat{\sum_{k=N/2}^{3N/4}}p^2_{N+k}$ can be uniformly estimated because of the cutoff introduced in (\ref{introcutoff}). 

Using the spectral gap estimate obtained in Appendix \ref{spectralgap} we can bound (\ref{R31}) by a constant, and thanks to the basic inequality
\[
\mathbb{E}_{\nu_y}\big[(g^{(2N)}\circ\sigma^{j,j+1}-g^{(2N)})^2\big] 
\leq
C\mathbb{E}_{\nu_y}\big[(X_{j,j+1}g^{(2N)})^2\big] \;,
\]
we can see after telescoping, that (\ref{R32}) is also uniformly bounded.
\end{proof}
\begin{lemma}[\textbf{Characterization of weak limits}]
\label{subsequencelimits}
Every weak limit function of the sequence $\{R^1_N - R^2_N\}_{N \geq 1}$ is of the form $cp_0p_1$ for some constant $c$.
\end{lemma}
\begin{proof}
Let us firstly consider the sequence $\{R^1_N\}_{N \geq 1}$. In Lemma \ref{boundingborderterms} we obtain a decomposition of $R^1_N$ as the sum of two terms, one of which converges to zero in $L^2(\nu_y)$. The other term, namely (\ref{secondterm}), is equal to $p_0p_1h^{1}_N(p_{0},\cdots,p_{-7N/2})$ where
\begin{equation}
\label{h1}
h^{1}_N\;=\;\widehat{\sum_{k=N/2}^{3N/4}}\frac{1}{2(N+k)y^2}\tau^{-N-k}\mathbb{E}_{\nu_y}[p^2_{-N-k-1}(p^2_{N+k+1}-y^2)g^{(2N)}|\mathfrak{F}^{N+k}_{-N-k}]\;.
\end{equation}
It was also proved that $\{p_0p_1h^{1}_N\}_{N \geq 1}$ is bounded in $L^2(\nu_y)$, therefore it contains a weakly convergent subsequence, say $\{p_0p_1h^{1}_{N'}\}_{N'}$. By similar arguments as in the proof of Lemma \ref{boundingborderterms}, we can conclude that $\{h^{1}_N\}_{N \geq 1}$ is bounded in $L^2(\nu_y)$, therefore $\{h^{1}_{N'}\}_{N'}$ contains a weakly convergent subsequence, whose limit will be denoted by $h^{1}$.

Applying the operator $X_{i+i+1}$ in the two sides of (\ref{h1}) and using Schwartz inequality, is easy to see that
\[
||X_{i,i+1}h^{1}_N||_{L^2(\nu_y)} \leq \frac{C}{N}||\xi||_{L^2(\nu_y)} \quad \quad \textit{for} \quad \quad \{i,i+1\} \subseteq \{0,-1,-2,\cdots\} \;,
\]
which implies that $X_{i,i+1}h^1=0$ for $\{i,i+1\} \subseteq \{0,-1,-2,\cdots\}$.
This, together with the fact that the function $h^{1}$ just depends on $\{p_0,p_{-1},p_{-2},\cdots\}$, permit to conclude that $h^{1}$ is a constant function, let's say $c$. Therefore $\{p_0p_1h^{1}_{N'}\}_{N'}$ converges weakly to $cp_0p_1$.

This proves that for every weakly convergent subsequence of $\{R^1_N\}_{N \geq 1}$ there exist a constant $c$ such that the limit is $cp_0p_1$. Exactly the same can be said about $\{R^2_N\}_{N \geq 1}$. 

Finally suppose that  $\{R^1_{N'} - R^2_{N'} \}_{{N'} \geq 1}$ is a  subsequence converging weakly to a function $f$. The boundness of  $\{R^1_{N'}\}_{{N'} \geq 1}$ and $\{R^2_{N'}\}_{{N'} \geq 1}$ implies the existence of further subsequences  $\{R^1_{N''}\}_{{N''} \geq 1}$ and $\{R^2_{N''}\}_{{N''} \geq 1}$ converging weakly to $c_1p_0p_1$ and $c_2p_0p_1$, respectively. Therefore, by unicity of the limit, we have $f = (c_1 - c_2)p_0p_1$.
\end{proof}

\begin{lemma}[\textbf{Convergence to zero}]
The sequence $\{R^1_N - R^2_N\}_{N \geq 1}$ converges weakly to zero.

\label{weakconvergence}
\end{lemma}  
\begin{proof} 
In view of the boundness of  $\{R^1_{N'}\}_{{N'} \geq 1}$ and $\{R^2_{N'}\}_{{N'} \geq 1}$, it is enough to prove that every weak limit of the sequence $\{R^1_N - R^2_N\}_{N \geq 1}$ is equal to zero.

At the end of the proof of Lemma \ref{subsequencelimits} we see that every weak limit of $\{R^1_{N} - R^2_{N} \}_{{N} \geq 1}$ is of the form $(c_1 - c_2)p_0p_1$, where $c_1$ and $c_2$ are constants for which there exist further subsequences  $\{R^1_{N'}\}_{N \geq 1}$ and $\{R^2_{N'}\}_{N \geq 1}$ converging weakly to $c_1p_0p_1$ and $c_2p_0p_1$, respectively. 

On the other hand, recall that 
\[
X_{0,1}\left(\sum_{j=-\infty}^{\infty}\tau^j\widehat{g}^{N}\right)=\xi+y^{-4}I^1_N+y^{-4}I^2_N+y^{-4}I^3_N+y^{-4}R^1_N-y^{-4}R^2_N.
\]
Let us multiply the two sides of the last equality by the function $p_0p_1 \in L^2(\nu_y)$, and take expectation with respect to $\nu_y$. Thanks to the orthogonality condition ii), namely $\mathbb{E}_{\nu_y}[p_0p_1\xi]=0$, we have
\[
0 \;=\; y^{-4}\mathbb{E}_{\nu_y}[ p_0p_1(I^1_{N'}+I^2_{N'}+I^3_{N'})] + y^{-4}\mathbb{E}_{\nu_y}[ p_0p_1(R^1_{N'}-R^2_{N'})].
\]
Finally, taking the limit as $N'$ tends to infinity we obtain that $0=c_1-c_2$.
\end{proof}

\appendix 

\section{Spectral Gap}
\label{spectralgap}
We investigate in this section the spectral gap for the dynamics induced by the infinitesimal generator given by
\begin{equation}
\label{generatorgap}
L_N(f)=\frac{1}{2}\sum_{x=1}^{N-1}X_{x,x+1}[a(p_x,p_{x+1})X_{x,x+1}(f)]\;,
\end{equation}
with associate Dirichlet form defined as
\begin{equation}
\label{dirichletgap}
D_N(f)=\frac{1}{2}\sum_{x=1}^{N-1}\int_{\mathbb{R}^{N}}a(p_x,p_{x+1})[X_{x,x+1}(f)]^2\nu_y^{N}(dp).
\end{equation}
It is enough to consider $a \equiv 1$ and  $y=1$, so we omit the subindex $y$ in $\nu_y^{N}$.

The idea will be to relate our model with a similar one, known as the Kac's model, whose spectral gap is already known. Specifically, we find a relation between their Dirichlet forms and use it to obtain the desired spectral gap estimate for our model. 

The generator of the Kac's model is defined for continuous functions as 
\begin{equation}
\label{kacgenerator}
\mathfrak{L}_N(f)=\frac{1}{C_2^N}\sum_{1\leq i\leq j\leq N}\frac{1}{2\pi} \int_0^{2\pi}[f(R^{\theta}_{i,j}x)-f(x)]d\theta\;,
\end{equation}
where $R^{\theta}_{i,j}$ represents a clockwise rotation of angle $\theta$ on the plane $i,j$.
It is easy to see that spheres are invariant under this dynamics. 

To this generator is associated the following Dirichlet form 
\begin{equation}
\label{kacdirichletform}
\mathfrak{D}_N(f)=\frac{1}{C_2^N}\sum_{1\leq i< j\leq N}\frac{1}{2\pi} \int_0^{2\pi}\int_{S_r^{N-1}}[f(R^{\theta}_{i,j}x)-f(x)]^2 d\sigma_r(x) d\theta\;,
\end{equation}
where $S_r^{N-1}$ is the (N-1)-dimensional sphere of radius $r$ centered at the origin and $\sigma_r$ stands for the uniform measure over this sphere. In order to study the spectral gap is enough to treat with the unitary sphere, in which case we omit the subindex. 

This dynamics was used by Kac as a model for the spatially homogeneous Boltzmann equation. A complete description of this model can be founded in \cite{Car}.

Let us state the spectral gap estimate obtained in \cite{Jan} for the Kac's model.
\begin{lemma}[Janvresse]
\label{janvresse}
There exist a constant $C$ such that for all $f\in L^2(S^{N-1})$ we have
\[
E_{\sigma}[f;f]\leq C\  N\  \mathfrak{D}_N(f)\;,
\]
where $E_{\sigma}[f;g]$ denotes the covariance between $f$ and $g$ with respect to $\sigma$.
\end{lemma}
Defining
\begin{equation*}
B_{i,j}f(x)=\frac{1}{2\pi}\int_{0}^{2\pi}[f(R^{\theta}_{i,j}x)-f(x)]^2d\theta \;
\end{equation*}
and using the identity $\frac{\partial}{\partial\theta}[f(R^{\theta}_{i,j}x)-f(x)]=-X_{i,j}(f)(R^{\theta}_{i,j}x)$ together with Poincar\'e inequality on the interval $[0,2\pi]$, we obtain a term by term relation between (\ref{dirichletgap}) and (\ref{kacdirichletform}). Namely,
\[
\int_{S^{N-1}}B_{i,j}f(x)d\sigma(x)\leq 2\pi \int_{S^{N-1}} |X_{i,j}(f)|^2d\sigma(x)\;.
\]
Observe that in (\ref{dirichletgap}) just near neighbors interactions are involved, while in (\ref{kacdirichletform}) long range interactions are also considered. This fact demands an  additional argument in order to relate the two Dirichlet forms.

Elementary calculations based on the symmetries of the measure $\sigma$ lead us to the so-called path lemma
\begin{equation}
\label{basic5}
\int_{s^{N-1}}B_{i,i+k}f(x)d\sigma(x)\leq
64k\sum_{j=0}^{k-2}\int_{s^{N-1}}B_{i+j,i+j+1}f(x)d\sigma(x)\;.
\end{equation}
Finally we state the main result of this section, which follows from the preceding inequality, Lemma \ref{janvresse}, and the next well known formula
\begin{align*}
\int_{\mathbb{R}^{N}}f(p) \nu^{N}(dp)&=\int_0^{\infty}\left(\int_{S^{N-1}(\rho)}f(x)d\sigma(x)\right)\frac{\omega_{N}(\rho)}{(\sqrt{2\pi})^{N}}e^{-\frac{\rho^2}{2}}d\rho \;,
\end{align*}
where $\omega_{N}$ denotes the surface area of $S^{N-1}$.
\begin{lemma}
There exists a positive constant $C$ such that, for every $f\in L^2(\mathbb R^N)$ satisfying $\int_{S^{N-1}(r)}fd\sigma=0$ for all $r>0$, we have
\[
\int_{\mathbb{R}^{N}}f^2 \nu^{N}(dp) \leq CN^2\sum_{x=1}^{N-1}\int_{\mathbb{R}^{N}}[X_{x,x+1}(f)]^2\Phi_{N}dp\;.
\]
\end{lemma}

\section{Some Geometrical Considerations}
\label{geomconsider}
The aim of this section is to establish conditions over a given set of functions  $\xi_{i,i+1}:\mathbb{R}^{n+1} \to \mathbb{R}$ for ${1\leq i\leq n}$ , which ensure the existence of a function $g:\mathbb{R}^{n+1} \to \mathbb{R}$ satisfying
\begin{equation}
\label{poissonequation}
X_{i,i+1}(g)=\xi_{i,i+1} \ \ \textit{for}\ \   1\leq i\leq n\;,
\end{equation}
where $X_{x,y}=p_{y}\partial_{p_{x}}-p_{x}\partial_{p_{y}}$. 

Observe that the vector fields $X_{i,j}$ act on spheres. In fact, we are interested in solving (\ref{poissonequation}) over spheres. The results obtained in this section will be used in the proof of Theorem \ref{characterizationofH_y} in Section \ref{{H}_y}. 

Let us remark that for $1\leq i<j \leq n$ we have
\begin{equation}
\label{liebrack}
[X_{i,i+1},X_{j,j+1}]=
\begin{cases}
X_{i,j+1}&\textit{if} \quad\quad i+1=j,\\
0&\textit{if} \quad\quad i+1 \neq j, 
\end{cases}
\end{equation}
where $[,]$ stands for the Lie bracket. Thus, the existence of such a function $g:\mathbb{R}^{n+1} \to \mathbb{R}$ satisfying (\ref{poissonequation}) give us some necessary conditions over the family $\{\xi_{i,i+1}\}$, namely
\begin{align}
\label{intcondition1}
X_{i,i+1}(\xi_{j,j+1})=X_{j,j+1}(\xi_{i,i+1}) \ \ \ \ \textit{if}\ \ \ \ i+1 \neq j\;,\\  \label{intcondition2}
p_{i+1}X_{j,j+1}(\xi_{i,i+1})-p_{i+1}X_{i,i+1}(\xi_{j,j+1})=p_{j+1}\xi_{i,i+1}\;+\;p_{i}\xi_{j,j+1}\;,
\end{align}
for $1\leq i<j \leq n$.

We state the main result of this section.
\begin{theorem}
\label{inttheorem}
Let $\xi_{i,i+1}:\mathbb{R}^{n+1} \to \mathbb{R}\;$ for ${1\leq i\leq n}$ be a  a set of functions satisfying conditions (\ref{intcondition1}) and (\ref{intcondition2}). Then, for every $r>0$ there exists a function $g_r:S^n(r) \to \mathbb R\;$  such that
\[
X_{i,i+1}(g_r)=\xi_{i,i+1} \ \ \textit{on} \ \ S^n(r)\;. 
\]
\end{theorem}
The approach we adopt to prove this result consist in defining over each sphere $S^n(r)$ an \textit{ad hoc} differential $1$-form $\omega_r$. Conditions (\ref{intcondition1}) and (\ref{intcondition2}) will imply the closeness of $\omega_r$, which in view of Proposition \ref{deRham} below implies the existence of the desired $g_r$.

Now we state two well known results in differential geometry.
\begin{proposition}
\label{differential1}
Let $\omega$ be a $1$-form on a differentiable manifold $M$, and $X,Y$ differential vector fields on $M$, then
\[
d\omega(X,Y)=X\omega(Y)-Y\omega(X)-\omega([X,Y])\;,
\]
where $[ \;,]$ denote the Lie bracket.
\end{proposition}
\begin{proposition}
\label{deRham}
Every closed differential 1-form on a n-dimensional sphere with $n\geq 2$ is exact.
\end{proposition}
\begin{proof}[Proof of Theorem \ref{inttheorem}]
Let $\{Y_1, \cdots, Y_n\}$ be a subset of $n$ vector fields belonging to the Lie algebra generated by $\{X_{i,i+1}\}_{1\leq i\leq n}$. Consider the subset of the sphere $A$ such that, for $p\in A\subset S^n(r)$ the set of vectors $\beta_{p}=\{Y_1(p), \cdots, Y_n(p)\}$ form a basis for the tangent space $T_pS^n(r)$, with associated dual basis denoted by $\beta^*_p=\{dY_1, \cdots, dY_n\}$.

Define on $A \subset S^n(r)$ the following differential 1-form
\[
\omega=\sum_{k=1}^n\xi_{k}\;dY_k.
\]
Here $\xi_{k}$ is the corresponding component in terms of $\{\xi_{i,i+1}\}_{1\leq i\leq n}$. For instance, if $Y_k=X_{1,2}$ then $\xi_{k}=\xi_{1,2}$, or if $Y_k=[X_{1,2},X_{2,3}]$ then $\xi_{k}=X_{1,2}(\xi_{2,3})-X_{2,3}(\xi_{1,2})$.  

It follows from (\ref{liebrack}) that the Lie algebra generated by $\{X_{i,i+1}\}_{1\leq i\leq n}$ is of maximal dimension over each sphere. Thus, varying over all the subsets of size $n$ of the Lie algebra, the sets $A$ form a covering of the whole sphere. Moreover, any two differential forms defined in this way will coincide in their common domain of definition. This induces a differential 1-form $\;\omega_r\;$ well defined on the whole sphere $S^n(r)$.

In order to prove closeness of the differential 1-form $\;\omega_r\;$, it suffices to prove closeness of each of the differential forms given above, which is reduce to prove $d\omega(Y_{i},Y_{j})=0\;$ for  $\; i,j \in \{1,\cdots,n\}$. The proof of this fact follows from conditions (\ref{intcondition1}), (\ref{intcondition2}) and Proposition \ref{differential1}. 
\end{proof}

\section{Equivalence of Ensembles}
\label{equiensem}
In this work we need to consider equivalence of ensembles for unbounded functions. The same is required in \cite{BBO} where, by means of a modification on the arguments in \cite{DF}, a proof of the following statement is given.
\begin{lemma}
Let $\nu_{N,y\sqrt N}$ be the uniform measure on the sphere 
$$
S^N(y\sqrt N)=\{(p_1,\cdots,p_N) \in \mathbb R^N : \sum_{i=1}^N p_i^2=Ny^2\}\;,
$$
and $\nu_y^{\infty}$ the infinite product of Gaussian measures with mean zero and variance $y^2$. Given a function $\phi$ on $\mathbb R^K$, such that for some positive constants $\theta$ and $C$ 
$$
|\phi(p_1,\cdots,p_K)| \leq C\left( \sum_{i=1}^K p_i^2 \right)^{\theta} \;,
$$
there exist a constant $C'=C'(C,\theta, K,y)$ such that
$$
\limsup_{N \to \infty}N \left|E_{\nu_{N,y\sqrt N}}[\phi]-E_{\nu_{y}^{\infty}}[\phi]\right| \leq C'\;.
$$
\end{lemma}

\bigskip

\textbf {Acknowledgements} 

\bigskip

It is a pleasure to thank my thesis advisors S. Olla and C. Landim  for suggesting this problem and for their valuable orientation. I am also grateful to  S.R.S.  Varadhan for his help in the proof of Theorem (\ref{characterizationofH_y}).

\end{document}